\documentclass[twocolumn,envcountsame]{svjour3}
\smartqed
\usepackage{amsmath,amssymb}
\usepackage{graphicx}
\usepackage{subcaption}
\usepackage{enumerate}
\usepackage{times}
\usepackage{hyperref}
\usepackage{algorithm,algpseudocode}
\usepackage{mathtools}
\usepackage{float}
\usepackage{pgfplots}
\usepackage{booktabs}

\usepackage{xcolor}
\usepackage[textsize=footnotesize,colorinlistoftodos]{todonotes}

\usepackage{tikz}

\newcommand{\RR}{\mathbf{R}}
\newcommand{\norm}[2][]{{\|{#2}\|_{}}_{#1}}
\newcommand{\abs}[1]{|#1|}
\newcommand{\scp}[2]{\langle #1,#2\rangle}

\newcommand{\laplace}{\Delta}

\newcommand{\dummy}{\,\cdot\,}
\newcommand{\pmat}[1]{\begin{pmatrix} #1 \end{pmatrix}}
\newcommand{\MA}{\mathcal{A}}
\newcommand{\MB}{\mathcal{B}}
\newcommand{\MR}{\mathcal{R}}
\newcommand{\udag}{u^{\dagger}}
\newcommand{\dd}{\,\mathrm{d}}
\newcommand{\subjectto}{\text{\quad s.t.\quad }}

\DeclareMathOperator{\diam}{diam}
\DeclareMathOperator{\id}{Id}

\DeclareMathOperator*{\argmin}{argmin}
\DeclareMathOperator{\ind}{\mathcal{I}}

\DeclareMathOperator{\prox}{prox}
\DeclareMathOperator{\proj}{proj}
\DeclareMathOperator{\gap}{gap}
\let\div\relax
\DeclareMathOperator{\grad}{\nabla}
\DeclareMathOperator{\div}{div}
\DeclareMathOperator{\symgrad}{\mathcal{E}}
\DeclareMathOperator{\Hess}{Hess}
\DeclareMathOperator{\TV}{TV}
\DeclareMathOperator{\TGV}{TGV}
\DeclareMathOperator{\CTGV}{CTGV}

\DeclareMathOperator{\DGTV}{DGTV}
\DeclareMathOperator{\DGTGV}{DGTGV}

\def\makeheadbox{\relax}

\title{Denoising of image gradients and total generalized variation denoising
\thanks{This material was based upon work partially supported by the National Science Foundation under Grant DMS-1127914 to the Statistical and Applied Mathematical Sciences Institute. 
    Any opinions, findings, and conclusions or recommendations expressed in this material are those of the author(s) and do not necessarily reflect the views of the National Science Foundation.}}
\author{B. Komander \and D. A. Lorenz \and L. Vestweber}
\institute{Birgit Komander\at Institute of
  Analysis and Algebra, TU Braunschweig, 38092 Braunschweig, Germany,
  \email{b.komander@tu-braunschweig.de}
\and D. A. Lorenz\at Institute of
  Analysis and Algebra, TU Braunschweig, 38092 Braunschweig, Germany,
  \email{d.lorenz@tu-braunschweig.de}
\and Lena Vestweber \at Institut 
  Computational Mathematics, AG Numerik, TU Braunschweig, 38092 Braunschweig, Germany,
  \email{l.vestweber@tu-braunschweig.de}}

\begin{document}
\maketitle

\begin{abstract}
  We revisit total variation denoising and study an augmented model where we assume that an estimate of the image gradient is available. We show that this increases the image reconstruction quality and derive that the resulting model resembles the total generalized variation denoising method, thus providing a new motivation for this model. Further, we propose to use a constraint denoising model and develop a variational denoising model that is basically parameter free, i.e. all model parameters are estimated directly from the noisy image.
  
  Moreover, we use Chambolle-Pock's primal dual method as well as the Douglas-Rachford method for
  the new models. For the latter one has to solve large discretizations of partial differential
  equations. We propose to do this in an inexact manner using the preconditioned conjugate gradients
  method and derive preconditioners for this. Numerical experiments show that the resulting method
  has good denoising properties and also that preconditioning does increase convergence speed
  significantly. Finally we analyze the duality gap of different formulations of the TGV denoising problem
  and derive a simple stopping criterion.
\end{abstract}


\section{Introduction}
\label{sec:intro}
In this work we revisit variational denoising of images with total variation penalties, dating back to the classical Rudin-Osher-Fatemi total variation denoising method~\cite{rudin1992rofmodel}.
We start by augmenting the model with an estimate of the image gradient and analyze, how this helps for image denoising.
This is related to the method of first estimating image normals and then using this estimate for a better image denoising, an approach proposed by Lysaker et al.~\cite{Lysaker2004}.
As we will see, a combined approach, which tries to estimate the  gradient of the denoised image and the denoised image itself simultaneously, is very close to the successful total generalized variation denoising from Bredies et al.~\cite{bredies2010total}.
A brief introduction to this idea was already proposed in~\cite{komander2017denoising}.
Further, we will propose different (in some sense equivalent) versions of the total general variation denoising method (one of these, CTGV, already introduced in~\cite{komander2017denoising}) which have several advantages over the classical one:
First, we are going to work with constraints in contrast to penalties, which, in some cases, allows for a simple, clean and effective parameter choice.
Second, different formulations of these problems lead to different dual problems and hence, different algorithms and some of these turn out to be a little simpler regarding duality gaps and stopping.
Moreover, the different models show slightly different numerical performance.
Finally, we will make use of the Douglas-Rachford method to solve these minimization problems.
This involves the solution of large discretizations of linear partial differential equations and we
will develop simple and effective preconditioners for these equations.
In contrast to~\cite{bredies2015preconditioned} where the authors use classical linear splitting methods for the inexact
solution of the linear equations, we propose to use a few iterations of the preconditioned conjugate gradient method.

The paper is organized as follows. Section~\ref{sec:two-stage-denoising} motivates denoising of
gradients as a mean to improve total variation denoising and derives several new variational methods. Then, section~\ref{sec:gaps} investigates the corresponding duality gaps of the problems and section~\ref{sec:numerics} deals with the numerical treatment and especially with the Douglas-Rachford method and efficient preconditioners for the respective linear subproblems. Section~\ref{sec:experiments} reports numerical experiments and section~\ref{sec:conclusion} draws some conclusions.

\section{Total variation denoising with estimates of the gradient}
\label{sec:two-stage-denoising}

Since its introduction in 1992, the Rudin-Osher-Fatemi model \cite{rudin1992rofmodel}, 
also known as total variation denoising, has found numerous applications.
One way to put this model is that the total variation of an image is used as a regularizer for an image denoising 
optimization problem, in general $\min_x F(x) + G(Kx)$, with $u_0$ as the input image, defined on a domain $\Omega$ as
\begin{align}\label{eq:ROF}%\tag{ROF}
\begin{split}
	&\min_u \lambda \int_{\Omega} \abs{\grad u}\dd x
		+ \frac{1}{2}\int_{\Omega} \abs{u(x)-u_0(x)}^2\dd x\\
	=&\min_u \lambda \norm[1]{\abs{\grad u}} 
		+ \frac{1}{2}\norm[2]{u-u_0}^2.
\end{split}	
\end{align}
One problem in the resulting denoised images is the occurring 
staircasing effect, i.e. the creation of flat 
areas separated by jumps.
One way to overcome this staircasing, proposed by Lysaker et al.~\cite{Lysaker2004}, 
is an 
image denoising technique in two steps.
There, in a first step, a 
total variation filter was used to smooth the normal vectors of the level 
sets of a given noisy image and then, as a second step,  
a surface was fitted to the resulting normal vectors.
The method was formulated in a dynamic way, i.e. by solving a certain partial differential equation to steady state.
A similar approach has been taken in~\cite{komander2014} for a problem of 
deflectometric surface measurements where the measurement device 
does not only produce approximate point coordinates but also 
approximate surface normals. 
It turned out that the incorporation of the surface normals 
results in an effective, but fairly complicated and non-linear problem.
Switching from surface normals to image gradients, however, turns 
the problem into a ``more linear'' one and leads to an equally effective method, see~\cite{komander2014}. 

In this section we follow the idea of introducing additional 
information, i.e. gradient information, into the ROF-model~\eqref{eq:ROF}
in order to prevent or reduce the  staircasing effect. 

\subsection{Denoising with prior knowledge on the gradient}
\label{sec:denoising-wit-gradient}

Consider the image model $u_0 = \udag + \eta$, where $u_0$ is the 
given noisy image, $\udag$ is the ground truth, i.e. the noise-free 
image, and $\eta$ is the additional Gaussian white noise.
In the situation of images, there are methods to obtain a reasonable estimate of the amount of noise, i.e. an estimate on $\norm[2]{\udag - u_{0}} = \norm[2]{\eta}$ is available.
One can use, for example the techniques from~\cite{liu2012noise,liu2013single} to estimate the noise level of Gaussian white noise quite accurately from a single image.
Using this information, it seems that 
\begin{align*}
	\norm[2]{u-u_0}\leq \norm[2]{\udag -u_0} = \norm[2]{\eta}
	\eqqcolon \delta_1
\end{align*}
is a sensible condition for the denoised image, since one should not look for an image $u$ further
away from $\udag$ than $u_{0}$. 
This motivates to consider a variant of the ROF model~\eqref{eq:ROF} where the discrepancy $\norm[2]{u-u_{0}}$ is not a penalty in the objective, but taken into account as a constraint.
This leads to a reformulation of the total variation problem as
\begin{align*}
	\min_u \norm[1]{\abs{\grad u}}\subjectto\norm[2]{u-u_0}
	\leq \norm[2]{\udag - u_0} = \norm[2]{\eta}.
\end{align*}
By estimating $\eta$ as Gaussian noise from the given image $u_{0}$ (e.g., using the method from~\cite{liu2012noise,liu2013single}) one obtains a parameter free denoising method.

Next, assume that we have some additional information on the original image $\udag$ available, namely some estimate $v$ of its gradient.
This could be taken into account as
\begin{align}
	\min_u\norm[1]{\abs{\grad u - v}}\subjectto 
	\norm[2]{u-u_0}\leq \norm[2]{\eta}.
\end{align}
It turns out, that this information can be quite powerful.
The next simple lemma shows that if we would know the gradient of $\udag$ and the noise level exactly, our model would recover $\udag$ perfectly, even for arbitrary large noise (and also independent of the type of noise).
\begin{lemma}
  \label{lem:perfect-reconstruction-udag}
  Assume that $\udag$ and $u_{0}$ fulfill $\int_{\Omega}\udag = \int_{\Omega}u_{0}$ and
  let $v = \grad \udag$ and $\delta_{1} = \norm[2]{\udag-u_{0}}$. Then it holds that
  \begin{equation}
    \label{eq:tvgrad-true-delta}
    \udag = \argmin_u \norm[1]{\abs{\grad u-v}}
    \subjectto\norm[2]{u-u_0}\leq \delta_{1},
  \end{equation}
  i.e. $\udag$ is the unique solution of the denoising problem.
\end{lemma}
\begin{proof}
  The set of minimizers is
  \begin{align*}
    \argmin_{u}\norm[1]{\abs{\grad u - \grad\udag}}
    \subjectto \norm[2]{u-u_0}
    \leq\norm[2]{\udag - u_0}.
  \end{align*}
  Clearly, $\udag$ is within this set, since the optimal value
  is $0$ and $\udag$ is feasible, because the constraint is 
  trivially fulfilled.
  
  To show that $\udag$ is indeed the unique solution, consider any other $u$ that also produces an objective value of zero.
  This implies $\grad u = \grad\udag$, i.e. $u = \udag + c$ for some constant $c$.
  Thus, $u$ fulfills the constraint $\norm[2]{u-u_{0}}^{2}\leq \norm[2]{\udag-u_{0}}^{2}$ if $\norm[2]{\udag-u_{0} + c}^{2}\leq \norm[2]{\udag-u_{0}}^{2}$.
  We expand the left hand side and get, writing $\abs{\Omega}$ for the measure of $\Omega$,
  \[
  \norm[2]{\udag-u_{0}}^{2} + 2c\int_{\Omega}(\udag-u_{0}) + c^{2}\abs{\Omega} \leq \norm[2]{\udag-u_{0}}^{2}.
  \]
  Since the middle integral vanishes by assumption, we see that $c=0$.\\
  \qed
\end{proof}

The next lemma shows, that $v=\grad\udag$ is also necessary for perfect reconstruction.
\begin{lemma}
  If $\delta_{1} = \norm[2]{\udag-u_{0}}$, $\int_{\Omega}\udag = \int_{\Omega}u_{0}$ and $\udag$ solves~\eqref{eq:tvgrad-true-delta}, then $v=\grad\udag$.
\end{lemma}
\begin{proof}
  Let
  $\udag\in\argmin_u \norm[1]{\abs{\grad u - v}}$ s.t.
  $\norm[2]{u-u_0}\leq\norm[2]{\udag - u_0}$.
  We denote by $\ind_C(x)$ the indicator function of a set $C$,
  and  set $K = \nabla$,
  $F(u) =
  \ind_{\norm[2]{\dummy-u_0}\leq\norm[2]{\udag - u_0}}(u)$, $G(q) = \norm[1]{\abs{q-v}}$
  (i.e. the Fenchel conjugate of $G$ is
  $G^{*}(\phi) \\= \int_{\Omega} v\phi\dd x +
  \ind_{\norm[1]{\abs{\dummy}}\leq 1}(\phi)$).
  The characterization of optimality by the Fenchel-Rockafellar optimality system~\cite[Remark 4.2]{ekeland1976convex}
  shows that $(u^{*},\phi^{*})$ is a primal-dual optimal pair if and only if
  \begin{equation*}
    \begin{cases}
      0 \in K^*\phi^* + \partial F(u^*),\\
      0 \in -Ku^* + \partial G^*(\phi^*),
    \end{cases}
  \end{equation*}
  which amounts to the inclusions
  \begin{align*}
    &\begin{cases}
      0 \in K^*\phi^* + \partial F(u^*),\\
      0 \in -Ku^* + \partial G^*(\phi^*),
    \end{cases}\\
    \Leftrightarrow 
    &\begin{cases}
      0 \in K^*\phi^* + \partial \ind_{\norm[2]{\dummy-u_0}
        \leq\norm[2]{\udag - u_0}}(u^*),\\
      0 \in -Ku^* + v + \partial \ind_{\norm[1]{\abs{\dummy}}
        \leq 1}(\phi^*).
    \end{cases}
  \end{align*}
  That means, that $(\udag,\phi^*)$ is optimal if and only if
  \begin{align*}
    \begin{cases}
      0 \in -\div\phi^* + \partial \ind_{\norm[2]{\dummy-u_0}
        \leq\norm[2]{\udag - u_0}}(\udag),\\
      0 \in -\grad\udag + v + \partial 
      \ind_{\norm[1]{\abs{\dummy}}
        \leq 1}(\phi^*).
    \end{cases}
  \end{align*}
  Since $\udag$ is on the boundary of the domain of the indicator function in the first inclusion, the subgradient there is the normal cone, which implies
  \begin{align*}
    &\begin{cases}
      \exists\, t\geq 0\,:\, 0 = \div\phi^* + t(\udag - u_0),\\
      \begin{cases}
        \grad \udag (x) = v(x),\\
        \qquad\text{if}\ \abs{\phi^*(x)} < 1,\\
        \exists\, s(x) \geq 0\,:\, \grad \udag(x) - v(x) = 
        s(x)\phi^*(x), \\
        \qquad\text{if}\ \abs{\phi^*(x)} = 1.
      \end{cases}
    \end{cases}
  \end{align*}
  By a result of Bourgain and Brezis~\cite[Proposition 1]{bourgain2003divYf} there is an $L^{\infty}$ solution $\phi$ of $-\div\phi = \udag - u_0$, i.e. there exists $L>0$ such that $\abs{\phi}\leq L$ a.e. and hence for $\tilde{\phi} = \phi/(L+1)$ it holds that $-\div
  \tilde{\phi} = \tfrac1{L+1}(\udag - u_0)$ and  $\abs{\tilde\phi}
  < 1$ a.e. Hence, $v = \grad \udag$ a.e.\hfill\qed
\end{proof}

\begin{remark}
  If the condition $\int_{\Omega}\udag = \int_{\Omega}u_{0}$ in Lemma~\ref{lem:perfect-reconstruction-udag} does not hold, but $\int_{\Omega}(\udag-u_{0}) = \varepsilon$,
  then the proof of that lemma still shows that all solutions of~\eqref{eq:tvgrad-true-delta} are of the form $\udag + c$ with $\abs{c-\tfrac{\varepsilon}{\abs{\Omega}}}\leq \tfrac{\varepsilon}{\abs{\Omega}}$.
\end{remark}

If $v\neq\grad \udag$, then any solution $\tilde u$ of \eqref{eq:tvgrad-true-delta} will usually be different from $\udag$,
although, it will fulfill the trivial estimate
\begin{align*}
  \norm[2]{\tilde u - \udag} & \leq\norm[2]{\tilde u-u_{0}} + \norm[2]{u_{0}-\udag}\\
                          &\leq 2\norm[2]{u_{0}-\udag} = \delta_1.
\end{align*}
However, the following lemma shows that $\tilde u\to \udag$ for $v\to\grad\udag$ (for constant noise level $\delta_{1}$):
\begin{lemma}
  Assume $\Omega$ is convex, $\udag$ and $u_{0}$ fulfill $\int_{\Omega}\udag = \int_{\Omega}u_{0}$, let $v$
  fulfill $\norm[1]{\abs{v - \grad \udag}}\leq \varepsilon$ and assume
  that there exists a solution $\tilde u$
  of~\eqref{eq:tvgrad-true-delta} with
  $\delta_{1} = \norm[2]{\udag-u_{0}}$ for which the constraint is
  active (i.e.  $\norm[2]{\tilde u-u_{0}} = \delta_{1}$).
  
  Then there exists another solution $\bar u$ of~\eqref{eq:tvgrad-true-delta}
  with $\delta_{1} = \norm[2]{\udag-u_{0}}$ that fulfills
  $\int_{\Omega}\bar u = \int_{\Omega}\udag$ and moreover, it holds
  that
  \[
  \norm[2]{\bar u-\udag}\leq \diam(\Omega)\varepsilon
  \]
  where $\diam(\Omega)$ denotes the diameter of $\Omega$.
\end{lemma}
\begin{proof}
  To obtain $\bar u$ we consider $\bar u = \tilde u + c$ for a
  suitable constant $c$. The equality
  $\int_{\Omega}\bar u = \int_{\Omega}\udag$ is achieved for
  $c = \int_{\Omega}(\udag - \tilde u)/\abs{\Omega}$. Since
  $\grad\bar u = \grad \tilde u$ holds, $\bar u$ is optimal
  for~\eqref{eq:tvgrad-true-delta} as soon as it is feasible. To check
  feasibility we calculate
  \begin{align*}
    \norm[2]{\bar u - u_{0}}^{2} & = \norm[2]{\tilde u - u_{0} + c}^{2}\\
    & = \norm[2]{\tilde u - u_{0}}^{2} + 2c\int_{\Omega}(\tilde u - u_{0}) + c^{2}\abs{\Omega}.
  \end{align*}
  Since
  \begin{align*}
  \int_{\Omega}(\tilde u-u_{0}) & = \int_{\Omega}(\tilde u - \udag) + \int_{\Omega}(\udag - u_{0})\\
    & = -c\abs{\Omega} + 0,
  \end{align*}
  we get that
  \begin{align*}
    \norm[2]{\bar u  - u_{0}}^{2} & = \norm[2]{\tilde u - u_{0}}^{2}-c^{2}\abs{\Omega}
  \end{align*}
  which shows feasibility of $\bar u$.
  
  Now we use the Poincar\'e-Wirtinger inequality in $L^{1}$ for which the optimal constant is known from~\cite{acosta2004optimal} to be $\diam(\Omega)/2$, i.e. it holds that
  \begin{align*}
    \norm[1]{\bar u - \udag} 
    &\leq \tfrac{\diam(\Omega)}{2}\norm[1]{\abs{\grad(\bar u - \udag)}} \\
    &\leq \tfrac{\diam(\Omega)}{2}(\norm[1]{\abs{\grad \bar u - v}} 
      + \norm[1]{\abs{v-\grad \udag}}).
  \end{align*}
  By optimality of $\bar u$ and feasibility of $\udag$ we get
  $\norm[1]{\abs{\grad \bar u - v}}\\
  \leq \norm[1]{\abs{\grad\udag -
      v}}$ and hence
  \[
  \norm[1]{\bar u - \udag} \leq  \diam(\Omega) \varepsilon.\hfill\qed
  \]
\end{proof}

\subsection{Denoising of image gradients}
\label{sec:denoising-of-gradients}

The previous lemma shows that any approximation $v$ of the true gradient $\grad\udag$ is helpful for total variation denoising according to~\eqref{eq:tvgrad-true-delta}.
In order to determine such a $v$, one way is to denoise 
the gradient of the input image, specifically, by a variational method
with a smoothness penalty for the gradient and some discrepancy term. 
Naturally, a norm of the derivate 
of the gradients can be used. A first candidate could be the 
Jacobian of the gradient, i.e.
\begin{align*}
	J(\grad u) = \pmat{
					\partial_1(\partial_1 u) 
					& \partial_2(\partial_1 u)\\
					\partial_1(\partial_2 u) 
					& \partial_2(\partial_2 u)
				}
\end{align*}
which amounts to the Hessian of $u$. Thus, the matrix is 
symmetric as soon as $u$ is twice continuously differentiable.
However, notice, that the Jacobian of an arbitrary vector field
is not necessarily symmetric and hence using $\norm{J(v)}$
as smoothness penalty seems unnatural. Instead, we could use the symmetrized Jacobian,
\begin{align*}
	\symgrad(v) = \pmat{
				\partial_1 v_1
				& \tfrac{1}{2}(\partial_1 v_2 + \partial_2 v_1)\\
				\tfrac{1}{2}(\partial_1 v_2 + \partial_2 v_1)
				& \partial_2 v_2
				},
\end{align*}
where $v_1$ and $v_2$ are the components of $v$. 
Note that for twice differentiable $u$ we have
\begin{align*}
	\symgrad(\grad u) = J(\grad u) = \Hess(u),
\end{align*}
i.e. in both cases we obtain the Hessian of $u$. 
Imitating the $\TV$-seminorm (and also followoing the idea of total generalized variation), we take
$F(v) = \norm[1]{\abs{\symgrad v}}$.

Similar to the constraint in~\eqref{eq:tvgrad-true-delta} the 
denoised gradient should not differ more from the true gradient
than $\grad u_0$, thus, we consider the minimization problem with a constraint 
\begin{align*}
	\min_v\norm[1]{\abs{\symgrad(v)}}\subjectto
		\norm[1]{\abs{\grad u_0 - v}}
		\leq \delta_{2}%c\norm[1]{\abs{\grad u_0}}.
\end{align*}
The parameter $\delta_{2}$ can be chosen as follows: If we set $\delta_{2} \coloneqq
c\norm[1]{\abs{\grad u_0}}$, then $c=1$ would allow the trivial minimizer $v=0$ and any $c<1$ will enforce some structure of $\grad u_{0}$ onto the minimizer and smaller $c$ leads to less denoising.

Putting the pieces together, we arrive at a two-stage denoising 
method:
\begin{enumerate}
\item	Choose $0<c<1$ and calculate a denoised gradient by solving
		\begin{align}\label{eq:DGTV_1}
		\begin{split}
			\hat{v} \in &\argmin_v \norm[1]{\abs{\symgrad(v)}}\\
				&\subjectto\norm[1]{\abs{\grad u_0 - v}}
				\leq c\norm[1]{\abs{\grad u_0}}\eqqcolon
				\delta_2.
		\end{split}
		\end{align}
\item Denoise $u_{0}$ by solving
		\begin{align}\label{eq:DGTV_2}
		\begin{split}
			\hat{u} \in &\argmin_u \norm[1]{\abs{\grad u - \hat{v}}}\\
				&\subjectto  \norm[2]{u-u_0}\leq 
				\norm[2]{\eta}\eqqcolon \delta_1.
		\end{split}
		\end{align}
\end{enumerate}
Instead of using the constrained formulation of the first 
problem, we can also use a penalized formulation. Thus, the 
gradient denoising problem writes as:
\begin{enumerate}
\item Choose $\alpha>0$ and calculate a denoised gradient by solving
		\begin{align}\label{eq:DGTGV_1}
			\hat{v} &= \argmin_v \norm[1]{\abs{\grad u_0 - v}}
				+ \alpha \norm[1]{\abs{\symgrad(v)}}
		\end{align}
\item Denoise $u_{0}$ by solving~\eqref{eq:DGTV_2}.
\end{enumerate}

Since we use a denoised gradient prior to apply total variation denoising, we term the method~\eqref{eq:DGTV_1} and~\eqref{eq:DGTV_2} Denoised Gradient Total Variation (DGTV).
Due to the similarity with total generalized variation, we call the second method based on~\eqref{eq:DGTGV_1} and~\eqref{eq:DGTV_2} Denoised Gradient Total Generalized Variation (DGTGV).

\subsection{Constrained and Morozov total generalized variation}
\label{sec:ctgv}

Both two-stage denoising methods DGTV and DGTGV for the gradient resemble previously known methods: The latter is related to total generalized variation (TGV)~\cite{bredies2010total} while the former to constrained total generalized variation (CTGV)~\cite{komander2017denoising}.
The TGV of second order, defined in~\cite{bredies2010total} has been shown to be equal to
\begin{align}\label{eq:TGV}
	\TGV^2_{(\alpha_0,\alpha_1)} (u) 
	= \min_v\, \alpha_1\norm[1]{\abs{\grad u - v}} + 
		\alpha_0\norm[1]{\abs{\symgrad(v)}}
\end{align}
in~\cite{bredies2013properties,bredies2010total} while  CTGV from~\cite{komander2017denoising} is defined as 
\begin{align}\label{eq:CTGV}
	\CTGV_{\delta} (u) 
	= \min_v\Big\{ \norm[1]{\abs{\symgrad(v)}}\subjectto \norm[1]{\abs{\grad u - v}}\leq\delta\Big\}.
\end{align}
Considering $\grad u$ to be the given data in these problems, one could say, following the notion from~\cite{lorenz2013neccond}, that TGV is a Tikhonov-type estimation of $\grad u$ while CTGV is a Morozov-type estimation of $\grad u$.

Now, combining the two steps of DGTGV into one optimization 
problem, where in each step the image as well as the gradient is 
updated simultaneously, we get 
\begin{align}\label{eq:mtgv}
\begin{split}
	&\min_{u} \Big\{\TGV^2_{(\alpha,1)} (u) 
		\subjectto \norm[2]{u-u_0}\leq \delta_1\Big\}\\
	=&\min_{u,v} \Big\{\norm[1]{\abs{\grad u - v}} 
		+ \alpha\norm[1]{\abs{\symgrad(v)}}\\
		&\qquad\qquad\subjectto \norm[2]{u-u_0}\leq \delta_1\Big\}.
\end{split}
\end{align}
This formulation is a Morozov type formulation of the TGV problem
\begin{equation}\label{eq:minTGV}
  \min_{u} \tfrac12\norm[2]{u-u_{0}}^{2} + \TGV^2_{(\alpha_{0},\alpha_{1})} (u) 
\end{equation}
and thus, in the following referred to as MTGV.

Another approach to form a combined optimization problem out of 
DGTGV is to preserve both constraints, i.e. taking~
\eqref{eq:DGTV_1} and~\eqref{eq:DGTV_2} to obtain
\begin{align}\label{eq:ctgv}
\begin{split}
	\min_{u,v}\, \norm[1]{\abs{\symgrad(v)}}
	\subjectto
		&\norm[2]{u-u_0}\leq \delta_1,\\
		&\norm[1]{\abs{\grad u - v}}\leq \delta_2.
\end{split}
\end{align}
In both problem formulations $\delta_1$ again is the 
noise level.
Using~\eqref{eq:CTGV} we see that
problem (\ref{eq:ctgv}) becomes
\begin{align}\label{eq:minCTGV}
	\min_u \, \CTGV_{\delta_2}(u) \subjectto 
	 \norm[2]{u-u_0}\leq\delta_1.
\end{align}

Obviously the TGV, MTGV and the CTGV denoising problems are implicitly
equivalent in the sense that, knowing the solution to one of the problems allows to calculate the respective parameters of one of the other problems such that the solution stays one (cf.~\cite[Theorem 2.3]{lorenz2013neccond} for a general result and~\cite[Lemma 1, Lemma 2]{komander2017denoising} for a result in the case of CTGV and TGV).

\subsection{Parameter choice}
\label{subsec:parameter-choice}

A few words on parameter choice for all methods are in order.
Frequently, a constraint
$\norm[2]{u-u_0}\leq \delta_1$ appears in the problem and
the parameter $\delta_{1}$ has a large influence on the denoising result.
A natural choice is to adapt the parameter to the noise level, i.e.
\begin{align*}
	\delta_1 = \norm[2]{\eta} = \norm[2]{u_{0}-u^{\dag}}.
\end{align*}
For a given discrete image, this number can be estimated as follows:
Under the assumption that the noise is additive Gaussian white noise, the methods
from~\cite{liu2012noise,liu2013single} allow to estimate the standard deviation $\sigma$ of the noise $\eta$.
Then, it is well known that the 2-norm of $\eta$ is estimated ny $\norm[2]{\eta} \approx  \sigma\sqrt{k}$, where $k$ is the number of pixels.

The second parameter in DGTV is $\delta_2 = c
\norm[1]{\abs{\grad u_0}}$, the constraint parameter 
in~\eqref{eq:DGTV_1}. 
Setting $c = 1$ leads to $v = 0$ as a feasible and also optimal 
solution, hence, the gradient we insert into the second 
optimization problem is zero and thus the second problem becomes 
pure total variation denoising without any additional 
information. Therefore, $c\in (0,1)$ is a reasonable choice.
Experiments showed (cf. section~\ref{sec:experiments}) that a lot 
of gradient denoising, i.e. smoothening of $\grad u_0$, leads to 
good reconstructions of the image in the second step. Thus, we set 
$c\approx 0.99$. 

The method DGTGV, the penalized variant of the two-stage method, includes the parameter $\alpha$, the penalization parameter 
in~\eqref{eq:DGTGV_1}. Again experiments showed that
$\alpha = 1$ leads to a good image reconstruction (cf. section
~\ref{sec:experiments}), independent of the noise level, the image size or the image type.

Numerical experiments on the performance of the two-stage methods with regard to image quality and
computational speed can be found in section~\ref{sec:experiments}.

Also the problems CTGV-, TGV-, and MTGV-denoising come with two parameters each that have to be chosen.
For CTGV- and MTGV-denoising the choice for $\delta_{1}$ is the same as above.
For the parameter $\alpha$ in MTGV-denoising~\eqref{eq:mtgv} we have experimental experience that hints that $\alpha \approx 2$ is a good universal parameter (cf.~section~\ref{sec:experiments}).
This inline with the usual recommendation that $\alpha_1 = 2\alpha_0$ is a good choice for TGV denoising from~\eqref{eq:minTGV}~cf.~\cite{KnollBrediesTGV2011}.
For the remaining parameter $\delta_{2}$ for CTGV there is following heuristic from~\cite{komander2017denoising}:
We denote with $u^{\TV}$ the $\TV$ denoised image with $\delta_{1}$ chosen according to the noise level estimate from Section~\ref{subsec:parameter-choice} and set $\delta_{2} =c\norm[1]{\abs{\grad u^{\TV}}}$ with $0<c<1$.
However, CTGV will not be included in the experiment in section~\ref{sec:experiments} and the main reason is, that the parameter choice here does lead to inferior results compared to MTGV.

\section{Duality gaps}
\label{sec:gaps} 
Besides the choice of the problem parameters, controlling the denoising, one has to choose a reasonable stopping criterion.
Since all considered problems are convex, the duality gap is a natural candidate.
Fenchel-Rockafellar optimality~\cite{ekeland1976convex} states that, under appropriate conditions, the primal problem 
$\min F(x) + G(Kx)$ has the dual problem \\
$\max_{y}-F^{*}(-K^{*}y) - G^{*}(y)$, that the duality gap, defined as
\[
\gap(x,y) = F(x) + G(Kx) + F^*(-K^{*}y) + G^{*}(y)
\]
is an upper bound on the difference from the current objective value to the optimal one, i.e. $\gap(x,y)\geq F(x) + G(Kx) - F(x^{*}) - G(Kx^{*})$ where $x^{*}$ is a solution of the primal problem, and that the duality gap vanishes exactly at primal-dual optimal pairs.

For the MTGV denoising problem~\eqref{eq:mtgv}
we have the gap function:
\begin{equation}
  \label{eq:gapmtgv}
  \begin{split}
    \gap&_{\operatorname{MTGV}}(u,v,p,q)=\\
    &\norm[1]{\abs{\grad u - v}} + 
    \alpha\norm[1]{\abs{\symgrad(v)}}
    +\ind_{\{\norm[2]{\cdot-u_0}\leq \delta_1\}}(u)\\
    &+\delta_1\norm[2]{\grad^* p} - \langle p, \grad u_0\rangle 
    +\ind_{\{0\}}(p - \symgrad^*q) \\
    &+\ind_{\{\norm[\infty]{\abs{\cdot}}\leq\alpha\}}(p) +
    \ind_{\{\norm[\infty]{\abs{\cdot}}\leq 1\}}(q).
  \end{split}
\end{equation}
As also noted in~\cite{bredies2015preconditioned} (for the TGV denoising problem~\eqref{eq:minTGV}) this gap is not helpful:
The problem is, that the indicator function $\ind_{\{0\}}(p - \symgrad^*q)$ is usually not finite as $p=\symgrad^{*}q$ is usually not fulfilled.
To circumvent this problem, one could simply replace $p$ by $\symgrad^*q$ in the gap function. 
If we do this, we obtain a gap that only depends on $q$:
\begin{equation}
  \label{eq:gapmtgvmod}
  \begin{split}   
    \gap&_{\operatorname{MTGV}}(u,v,q)=\\
    &\norm[1]{\abs{\grad u - v}} + 
    \alpha\norm[1]{\abs{\symgrad(v)}}
    +\ind_{\{\norm[2]{\cdot-u_0}\leq \delta_1\}}(u)\\
    &+\delta_1\norm[2]{\grad^* \symgrad^*q} - \langle \symgrad^*q, \grad u_0\rangle \\
    &+\ind_{\{\norm[\infty]{\abs{\cdot}}\leq\alpha\}}(\symgrad^*q) +
    \ind_{\{\norm[\infty]{\abs{\cdot}}\leq 1\}}(q).
  \end{split}
\end{equation}
We still have the problem, that 
$\ind_{\{\norm[\infty]{\abs{\cdot}}\leq\alpha\}}(\symgrad^*q)$ might be infinite.
So replacing $q$ and $p$ in $\eqref{eq:gapmtgv}$ by
\begin{align*}
  \tilde q&\coloneqq\frac{q}{\max\left(1,\frac{\norm[\infty]{\abs{\symgrad^*q}}}{\alpha}\right)}\\
  \tilde
  p&\coloneqq\frac{\symgrad^*q}{\max\left(1,\frac{\norm[\infty]{\abs{\symgrad^*q}}}{\alpha}\right)}=\symgrad^*\tilde
     q,
\end{align*}
we obtain a finite gap function
which is a valid stopping criterion for the problem:
\begin{lemma}\label{le:gapmtgv}
  Let $\Phi(u,v)\coloneqq\norm[1]{\abs{\grad u - v}} + 
  \alpha\norm[1]{\abs{\symgrad(v)}}
  +\ind_{\{\norm[2]{\cdot-u_0}\leq \delta_1\}}(u)$ be the primal functional of the
  MTGV denoising problem and $(u^*,v^*)$ be a solution of $\min_{u,v}\Phi(u,v)$. Then the
  error of the primal energy can be estimated by
  \begin{equation*}
    \begin{split}
      \Phi&(u,v)-\Phi(u^*,v^*)\\
      \leq &\norm[1]{\abs{\grad u - v}} + 
      \alpha\norm[1]{\abs{\symgrad(v)}}
      +\ind_{\{\norm[2]{\cdot-u_0}\leq \delta_1\}}(u)\\
      &+\delta_1\norm[2]{\grad^* (\symgrad^*\tilde q)} - \langle \symgrad^*\tilde q, \grad
      u_0\rangle + \ind_{\{\norm[\infty]{\abs{\cdot}}\leq 1\}}(\tilde q),
    \end{split}
  \end{equation*}
  where $\tilde
  q\coloneqq\frac{q}{\max\left(1,\frac{\norm[\infty]{\abs{\symgrad^*q}}}{\alpha}\right)}$ and $q$ is any feasible dual variable, i.e. $\norm[\infty]{\abs{q}}\leq 1$.
  Moreover, if $(u^{n},v^{n},q^{n})$ converge to a primal-dual solution, the upper bound converges to zero.
\end{lemma}
\begin{proof}
  The estimate is clear since any dual value is smaller or equal than any primal value. So we can change
  the dual variables in the gap as we like.
  
  Since the gap function is continuous on its domain, the sequence $(u^{n},v^{n},q^{n})$ converges to a primal-dual solution and has the same limit as $(u^{n},v^{n},\tilde q^{n})$, the second claim follows.\qed
\end{proof}
The remaining indicator functions that are left in the modified gap are guaranteed to be finite by several algorithms (namely ones that use projections onto the respective constraints).
As we will see later, this holds, for example, for classical methods like
Douglas-Rachford and Chambolle-Pock.

Similar problems with an infinite duality gap appear in other problems, too, e.g. in the CTGV denoising~\eqref{eq:minCTGV}.
A closer look at the  MTGV denoising problem \eqref{eq:mtgv} reveals that the above construction of the modified gap can indeed be avoided by a different choice of variables and that this holds for a broad class of problems.
First, we illustrate this for the MTGV problem. 
The MTGV problem does not change if we introduce a new variable $w=\grad u -v$ and replace $v$ in the
formulation. We obtain:
\begin{equation}\label{eq:mtgvvarchange}
\begin{split}
	\min_{u,w}\, \norm[1]{\abs{w}} + \alpha\norm[1]{\abs{\symgrad(\grad u - w)}}
	\subjectto 
	&\norm[2]{u-u_0}\leq \delta_1.
\end{split}
\end{equation}
For this problem, the gap function is
\begin{equation}
  \label{eq:gapmtgvvarchange}
  \begin{split}
    \gap&_{\operatorname{MTGV}}(u,w,q)=\\
    &\norm[1]{\abs{w}} + 
    \alpha\norm[1]{\abs{\symgrad(\grad u -w)}}
    +\ind_{\{\norm[2]{\cdot-u_0}\leq \delta_1\}}(u)\\
    &+\delta_1\norm[2]{\grad^* (\symgrad^*(q))} - \langle \symgrad^*q, \grad u_0\rangle \\ 
    &+
    \ind_{\{\norm[\infty]{\abs{\cdot}}\leq\alpha\}}(\symgrad^*q) +
    \ind_{\{\norm[\infty]{\abs{\cdot}}\leq 1\}}(q),
  \end{split}
\end{equation}
which is exactly the same as $\eqref{eq:gapmtgvmod}$. As above the indicator function 
$\ind_{\{\norm[\infty]{\abs{\cdot}}\leq\alpha_0\}}(\symgrad^*q)$ might be
infinite and therefore $q$ should be replaced by $\tilde q$ from above. 
So replacing $p$ by $\symgrad^*q$ in $\eqref{eq:gapmtgv}$ results in the gap function of
$\eqref{eq:mtgvvarchange}$. Both gap functions can be used as valid stopping criteria for
both problem formulations.

This method for the stopping criterion does apply to general problems of the form
\begin{align}
  \label{eq:minprob}
  \begin{split}
    \min_{u,v} F(u) + G(Av) + H(Bu-v),
  \end{split}
\end{align}
where $A$ and $B$ are linear (and standard regularity conditions, implying Fenchel-Rockafellar duality is fulfilled). 
The dual problem is
\begin{align*}
  \max_{p,q} -F^*(-B^*p) - H^*(p) - G^*(q) + \ind_{\{0\}}(p - A^*q).
\end{align*}
  If we replace $p$ by $A^*q$, we obtain
  \begin{align*}
    \max_q -F^*(-B^*A^*q) - H^*(A^*q) - G^*(q),
  \end{align*}
  which is the dual problem of $\eqref{eq:minprob}$ with variable change $w=Bu-v$:
  \begin{align*}
    \min_{u,w}	F(u) + G(A(Bu-w)) + H(w).
  \end{align*}

  The TGV denoising problem for example is very similar to MTGV. The gap is:
   \begin{align*}
     \gap&_{\operatorname{TGV}}(u,v,p,q)=\\
   &\alpha_1\norm[1]{\abs{\grad u - v}} + 
		\alpha_0\norm[1]{\abs{\symgrad(v)}}
		+\frac{1}{2}\norm[2]{u-u_0}^2\\
		&+\frac{1}{2}\norm[2]{\grad^* p}^2 - \langle p, \grad u_0\rangle 
		+\ind_{\{0\}}(p - \symgrad^*q) \\
		&+\ind_{\{\norm[\infty]{\abs{\cdot}}\leq\alpha_0\}}(p) +
		\ind_{\{\norm[\infty]{\abs{\cdot}}\leq\alpha_1\}}(q).
     \end{align*}
     With $\tilde p$ and $\tilde q$ as before one gets a simple gap for TGV:
   \begin{align*}
     \gap&_{\operatorname{TGV}}(u,v,\tilde q)=\\
           &\alpha_1\norm[1]{\abs{\grad u - v}} + 
		\alpha_0\norm[1]{\abs{\symgrad(v)}}
		+\frac{1}{2}\norm[2]{u-u_0}^2\\
		&+\frac{1}{2}\norm[2]{\grad^* (\symgrad^*\tilde q)}^2 - \langle \symgrad^*\tilde q, \grad
		u_0\rangle + \ind_{\{\norm[\infty]{\abs{\cdot}}\leq\alpha_1\}}(\tilde q).
	      \end{align*}
 \begin{corollary}\label{le:gaptgv}
	Let $\Phi(u,v)\coloneqq\norm[1]{\abs{\grad u - v}} + 
		\alpha\norm[1]{\abs{\symgrad(v)}}
		+\frac{1}{2}\norm[2]{u-u_0}^2$ be the primal functional of the
		TGV denoising problem and $(u^*,v^*)$ be a solution of $\min_{u,v}\Phi(u,v)$. Then the
		error of the primal energy can be estimated by
		\begin{equation*}
		  \begin{split}
		  \Phi&(u,v)-\Phi(u^*,v^*)\\
		  &\leq 
	  \alpha_1\norm[1]{\abs{\grad u - v}} + 
		\alpha_0\norm[1]{\abs{\symgrad(v)}}
		+\frac{1}{2}\norm[2]{u-u_0}^2\\
		&+\frac{1}{2}\norm[2]{\grad^* (\symgrad^*\tilde q)}^2 - \langle \symgrad^*\tilde q, \grad
		u_0\rangle + \ind_{\{\norm[\infty]{\abs{\cdot}}\leq\alpha_1\}}(\tilde q),
	      \end{split}
		\end{equation*}
		where $\tilde
		q\coloneqq\frac{q}{\max\left(1,\frac{\norm[\infty]{\abs{\symgrad^*q}}}{\alpha}\right)}$.
  Moreover, if $(u^{n},v^{n},q^{n})$ converge to a primal-dual solution, the upper bound converges to zero.
  \end{corollary}

\section{Numerics}
\label{sec:numerics}

In this section we describe methods to solve the convex optimization problems related to the various denoising methods from the previous sections.
We will work with standard discretizations of the images and the derivative operators, but state
them for the sake of completeness in Appendix~\ref{sec:discretization}.
In this section we focus on the optimization methods.

\subsection{Douglas-Rachford's method}
\label{sec:dr}
The Douglas-Rachford algorithm (see \cite{douglasrachford1956} and \cite{lions1979splitting}) is a splitting algorithm to solve monotone inclusions of the form
$0\in\MA(x)+\MB(x)$, which requires only the resolvents
$\MR_{t\MA} = (I+t\MA)^{-1}$ and $\MR_{t\MB}=(I+t\MB)^{-1}$ but not the resolvent of the sum
$\MA+\MB$.

One way to write down the Douglas-Rachford iteration is as a fixed point iteration
\begin{align*}
  z^k&=F(z^{k-1}),\quad\text{ with}\\
  F(z)&= z+\MR_{t\mathcal{A}}(2\MR_{t\MB}(z)-z)-\MR_{t\MB}(z)
\end{align*}
for some step-size $t>0$.
It is possible to employ relaxation for the Douglas-Rachford iteration as
\begin{equation}
  \label{DR}
  z^k=z^{k-1}+\rho(F(z^{k-1})-z^{k-1}),
  \end{equation}
  where $1<\rho<2$ is overrelaxation and $0<\rho<1$ is
  underrelaxation. For $\rho$ in the whole range from $0$ to $2$ and any $t>0$ it
  holds that the iteration converges to some fixed point $z$ of $F$
  such that $\MR_{t\MB}(z)$ is a zero of $\MA+\MB$, see,
  e.g.~\cite{eckstein1992douglas}.

 The Douglas-Rachford method can be used to solve the saddle point problem
\begin{align}\label{eq:saddleprob}
\begin{split}
	\min_{x}\max_y\ \langle y, K x\rangle + F(x) - G^*(y).
\end{split}
\end{align}
To that end, in~\cite{oconnor2014primal} the authors propose to use the splitting
\begin{align}
  \label{split1}
  \begin{split}
  \MB(x,y)&=\left[\begin{tabular}{c}
    $\partial F(x)$\\
    $\partial G^*(y)$
  \end{tabular}\right],\\
  \MA(x,y)&=
  \left[\begin{tabular}{l r}
    $0$&$K^*$\\
    $-K$&$0$\end{tabular}\right]
    \left[\begin{tabular}{c}
    $x$\\
    $y$\end{tabular}\right].
  \end{split}
\end{align}
If $\MA$ is a matrix the resolvent $\MR_{t\MA}$ simply is the matrix inverse 
$(I+t\MA)^{-1}$. The resolvent of $\MB$ is given by the proximal
mappings of $F$ and $G^*$, namely
\begin{align*}
  \MR_{t\mathcal{B}}(x,y)=\left[\begin{array}{c}(I+t\partial F)^{-1}(x)\\ (I+t\partial
    G^*)^{-1}(y)\end{array}\right]=\left[\begin{array}{c}\prox_{tF}(x)\\\prox_{tG^*}(y)\end{array}\right].
\end{align*}
For further flexibility, it is proposed in \cite{oconnor2014primal} to rescale the problem by
replacing $G$ with $\tilde{G}(y)\coloneqq G(\frac{y}{\beta})$ and
$K$ with $\tilde{K}\coloneqq \beta K$. Then one can replace the dual variable by $\tilde{y}^k\coloneqq\beta
y^k$ and finally use
$\prox_{t\tilde{G}^*}(y)=\frac{1}{\beta}\prox_{t\beta^2G^*}(\beta y)$ to obtain a second
stepsize $s\coloneqq \beta^2 t$, which can be chosen independently from the stepsize $t$.
In total, the resolvents then become
\begin{align*} 
  \MR_{t\mathcal{B}}(x,y)=\left[\begin{array}{c}\prox_{tF}(x)\\\prox_{sG^*}(y)\end{array}\right],\\
    \MR_{t\MA}= \left[\begin{tabular}{l r}
    $I$&$tK^*$\\
    $-sK$&$I$\end{tabular}\right]^{-1}.
\end{align*}
In each step of Douglas-Rachford we need the inverse of a fairly large block matrix. However, as
also noted in~\cite{oconnor2014primal} this can be done efficiently with the help of the Schur
complement. If $K$ is a matrix, $K^*=K^T$ and
\begin{align*}
  &\left[\begin{array}{cc}I&tK^T\\-sK&I\end{array}\right]^{-1}\\
    &=\left[\begin{array}{cc}0&0\\0&I\end{array}\right]+\left[\begin{array}{c}I\\sK\end{array}\right](I+stK^TK)^{-1}\left[\begin{array}{c}I\\-tK\end{array}\right]^T.
\end{align*}
To use this, we need to solve
equations with the positive definite coefficient matrix $I+\lambda K^TK$ efficiently. 

In the following we interpret the linear operators as matrices, without changing notation, i.e. 
the matrix $\symgrad$ also stands for the matrix realizing the linear operation $\symgrad$. With the help of the vectorization operation
   $\operatorname{vec}$, it holds that $\symgrad \cdot \operatorname{vec}(u)=\operatorname{vec}(\symgrad u)$, where
   on the left hand side $\symgrad$ is a matrix and on the right hand side $\symgrad$ is the linear operator from
   \ref{sec:discretization}.

As discussed is section~\ref{sec:gaps}, we have different possibilities to choose the primal variables (and thus, also for the dual variables):
Namely, we could use the primal variables $u$ and $v$, as, e.g. in the MTGV problem~\eqref{eq:mtgv}, or the primal variables $u$ and $w$, as in the corresponding problem~\eqref{eq:mtgvvarchange}.
This choice does not only influence the dual problem and the duality gap, but also the involved linear map $K^{T}K$.
The formulation with $u$ and $v$ as primal variables leads to
\begin{align}\label{eq:KTK-uv}
  K^TK=\begin{pmatrix}-\div&0\\-I&\mathcal{E}^T\end{pmatrix}
    \begin{pmatrix}\nabla&-I\\0&\mathcal{E}\end{pmatrix}
      =\begin{pmatrix}-\Delta&\div\\-\nabla&I+\mathcal{E}^T\mathcal{E}\end{pmatrix}
\end{align}
while the variable change to $u$ and $w$ gives
\begin{align}
  K^TK&=	\begin{pmatrix}-\operatorname{div}(\mathcal{E}^T)\\-\mathcal{E}^T\end{pmatrix}
    	\begin{pmatrix}\mathcal{E}(\nabla)&-\mathcal{E}\end{pmatrix}
	=\begin{pmatrix}H^T\\-\mathcal{E}^T\end{pmatrix}
    	\begin{pmatrix}H&-\mathcal{E}\end{pmatrix}\nonumber\\
	  &=\begin{pmatrix}H^TH&-H^T\mathcal{E}\\-\mathcal{E}^TH&\mathcal{E}^T\mathcal{E}\end{pmatrix},\label{eq:KTK-uw}
\end{align}
 where $H\coloneqq\mathcal{E}(\nabla)$ is the Hessian matrix, cf. Appendix~\ref{sec:proxdual}..
The same is true for CTGV from~\eqref{eq:ctgv}  and the usual TGV problem~\eqref{eq:minTGV}.
For problem~\eqref{eq:DGTGV_1} we have
\begin{align}
  K^TK&=\symgrad^T\symgrad.\label{eq:KTK-v}
\end{align}

 \subsection{Inexact Douglas-Rachford}
 \label{sec:idr}

 For the solution of the linear equation with coefficient matrix
 $I + \lambda K^TK$ we propose to use the preconditioned conjugate
 gradient method (PCG). For preconditioning we do several
 approximations: First we replace complicated discrete difference operators
 by simpler one and then we use a block-diagonal preconditioner and
 use the incomplete Cholesky decomposition in each block. As we will see, with these
 preconditioners we only need one or two iterations of PCG to obtain
 good convergence results of the Douglas-Rachford iteration.  If we
 only use one iteration of PCG and denote $A=(I+stK^TK)$, the linear
 step is approximated by
 \begin{align*}
   A^{-1}b\approx x_k -a M(Ax_k+b),
 \end{align*}
 where the PCG stepsize is given by 
 \begin{align*}
   a=\frac{\langle b-Ax_k,M(b-Ax_k)\rangle}{\langle
 AM(b-Ax_k),M(b-Ax_k)\rangle}
 \end{align*}
 and $M$ is the preconditioner for $A$, i.e. $M^{-1}\approx A$.
 
 In the following the coefficient matrices and preconditioners are given in
 detail.
With the matrices $D_1$ and $D_2$ representing the derivatives in the first and second
direction,
we have
\begin{align*}
  \grad =\begin{pmatrix}D_1\\D_2\end{pmatrix},\quad
  \mathcal{E}=\begin{pmatrix}D_1 & 0\\ \frac{1}{2}D_2 & \frac{1}{2}D_1 \\ 
	\frac{1}{2}D_2 & \frac{1}{2}D_1 \\ 
	0 & D_2\end{pmatrix},
\end{align*}
\begin{align*}
  J = \begin{pmatrix}D_1 & 0\\ D_2 & 0 \\ 
	0 & D_1 \\ 
	0 & D_2\end{pmatrix},\quad
  H = \begin{pmatrix}D_1^2 \\ D_1D_2 \\ 
	D_1D_2  \\ 
	D_2^2\end{pmatrix},
\end{align*}
\begin{align*}
  -\Delta = \nabla^T\nabla = D_1^TD_1 + D_2^TD_2,
\end{align*}
\begin{align*}
  \mathcal{E}^T\symgrad=\begin{pmatrix} D_1^TD_1+\frac{1}{2}D_2^TD_2 & \frac{1}{2}D_2^TD_1\\
    \frac{1}{2}D_1^TD_2 & \frac{1}{2}D_1^TD_1+D_2^TD_2\end{pmatrix},
\end{align*}
\begin{align*}
  J^TJ=\begin{pmatrix} D_1^TD_1+D_2^TD_2 & 0\\
    0 & D_1^TD_1+D_2^TD_2\end{pmatrix},
\end{align*}
\begin{align*}
	H^TH=(D_1^2)^TD_1^2+2(D_1D_2)^T(D_1D_2)+(D_2^2)^TD_2^2,
\end{align*}
\begin{align*}
  -\symgrad^TH=\begin{pmatrix}-D_1^TD_1^2-D_2^T(D_1D_2)\\-D_1^T(D_1D_2)-D_2^TD_2^2\end{pmatrix}.
\end{align*}

The operator $J^{T}J$ is in fact the negative (discrete) Laplace operator applied component-wise and also called vector Laplacian.
Note that the operator $\symgrad^{T}\symgrad$ decomposes as
\[
\symgrad^{T}\symgrad = \tfrac12 J^{T}J  + \tfrac12
\begin{bmatrix}
  D_{1}^{T}\\D_{2}
\end{bmatrix}
\begin{bmatrix}
  D_{1} D_{2}
\end{bmatrix}.
\]
Note that the boundary conditions are implicitely contained in the discretization operators and adjoints (e.g. we use Neumann boundary conditions for the gradient and thus, equations $-\laplace u = f$ always contain Neumann boundary conditions).

Our linear operators are dicretizations of continuous differential
operators. Hence, the resolvent steps correspond to solutions of certain differential equations.
In this context it has been shown to be beneficial, to motivate preconditioners for the linear systems by their continuous counterparts, see~\cite{mardal2011preconditioning}.
 Block diagonal
preconditioners are a natural choice for these operators. 
The conjugate gradient method is usually used for linear
systems of the form $\mathcal{A}x=f$, where $f$ is an element of a
finite dimensional space $X$. As discussed in
\cite{mardal2011preconditioning} it can also be used in the case where
$X$ is infinite dimensional and $\mathcal{A}:X\to X$ is a symmetric
and positive-definite isomorphism. Consider for example the negative
Laplace operator $\mathcal{A}: X=H_0^1(\Omega)\to X^*$ defined by
\begin{align*}
  \langle \mathcal{A} u,v\rangle = \int_{\Omega}\nabla u\cdot \nabla v\;dx,\quad u,v\in X.
\end{align*}
The standard weak formulation of the Dirichlet problem is now $\mathcal{A}x=f$, where $f\in X^*$. Since
$X\neq X^*$ the linear operator $\mathcal{A}$ maps $x$ out of the space and the conjugate
gradient method is not well defined. To overcome this problem, we introduce a preconditioner
$\mathcal{M}$, which is a symmetric and positive definite isomorphism mapping $X^*$ to $X$. The
preconditioned system $\mathcal{M}\mathcal{A}x=\mathcal{M}f$ can then be solved by the conjugate
gradient method. 
We consider now the corresponding continuous linear operator from~\eqref{eq:KTK-uv}:
\begin{align*}
  \MA_{u,v}=I+stK^*K = I + st\begin{pmatrix}-\Delta & -\nabla^*\\-\nabla &
    I+\symgrad^*\symgrad\end{pmatrix}.
\end{align*}
The operator $\MA_{u,v}$ is an isomorphism mapping $X=H^1(\Omega)\times (H^1(\Omega))^2$ into its dual 
$X^*=H^{1}(\Omega)'\times (H^{1}(\Omega)')^2$. 
The canonical
choice of a preconditioner, in the sense of~\cite{mardal2011preconditioning}, is therefore given as the blockdiagonal operator
\begin{align*}
  \mathcal{M}_{u,v}=\begin{pmatrix} (I-st\Delta)^{-1} & 0\\ 0 &
    \left(I+ st\left(I
    + J^*J\right)\right)^{-1}\end{pmatrix}
\end{align*}
(the inverses of the respective operators also exist in the continuous setting, see~\cite[Theorem 6.6]{alt2016linfuncana}). 
To see that $\MA_{u,v}$ has a bounded inverse, we can check coercivity, i.e. for some
$k>0$ we have 
\begin{align*}
 &\left\langle \MA_{u,v}(u,v),(u,v)\right\rangle\\
 &=\langle u,u\rangle_{L^2(\Omega)} + st\langle\nabla u,\nabla u\rangle_{(L^2(\Omega))^2}\\
 &\quad-2st\langle \nabla u, v\rangle_{(L^2(\Omega))^2}\\
 &\quad+(1+st)\langle v, v\rangle_{(L^2(\Omega))^2} + st\langle
  \symgrad v,\symgrad v\rangle_{(L^2(\Omega))^4}\\
 &=\|u\|^2_{L^2(\Omega)} + st\|\nabla u\|^2_{(L^2(\Omega))^2}
 -2st\langle \nabla u, v\rangle_{(L^2(\Omega))^2} \\
 &\quad +(1+st)\|v\|^2_{(L^2(\Omega))^2} + st\|\symgrad v\|^2_{(L^2(\Omega))^4}\\
 &\geq\|u\|^2_{L^2(\Omega)} + st\|\nabla u\|^2_{(L^2(\Omega))^2}\\
 &\quad-2st\|\nabla u\|_{(L^2(\Omega))^2}\|v\|_{(L^2(\Omega))^2} \\
 &\quad+(1+st)\|v\|^2_{(L^2(\Omega))^2} + st\|\symgrad v\|^2_{(L^2(\Omega))^4}\\
 &\geq\|u\|^2_{L^2(\Omega)} + st\|\nabla u\|^2_{(L^2(\Omega))^2}\\
 &\quad-2st\left(\frac{\|\nabla u\|^2_{(L^2(\Omega))^2}}{2\varepsilon}+\frac{\varepsilon\|v\|^2_{(L^2(\Omega))^2}}{2}\right) \\
 &\quad+(1+st)\|v\|^2_{(L^2(\Omega))^2} + st\|\symgrad v\|^2_{(L^2(\Omega))^4}\\
 &=\|u\|^2_{L^2(\Omega)} + st(1-\frac{1}{\varepsilon})\|\nabla u\|^2_{(L^2(\Omega))^2}\\
 &\quad+(1+st(1-\varepsilon))\|v\|^2_{(L^2(\Omega))^2} + st\|\symgrad v\|^2_{(L^2(\Omega))^4}\\
 &\geq \tilde k\left(\|u\|^2_{L^2(\Omega)} + \|\nabla
 u\|^2_{(L^2(\Omega))^2}\right)\\
 &\quad + \tilde k\left(\|v\|^2_{(L^2(\Omega))^2}
 + \|\symgrad v\|^2_{(L^2(\Omega))^4}\right)\\
 &\geq k\left(\|u\|_{H^1(\Omega)}^2+\|v\|^2_{(H^1(\Omega))^2}\right).
\end{align*}
The second to last estimate holds for some $\tilde k>0$ since for all $st>0$ we can find $\varepsilon>0$, such that
$st\left(1-\frac{1}\varepsilon\right)>0$ and $(1+st(1-\varepsilon))>0$. We can chose $\tilde k$ as the
minimum of both of them. The last estimate follows for some $k>0$ as a consequence of Korn's
inequality~\cite{ciarlet2013linear}.

The corresponding continuous linear operator from~\eqref{eq:KTK-v} is
  \begin{align*}
    \MA_{v}= I+st\symgrad^*\symgrad.
  \end{align*}
The operator $\MA_{v}$ is an isomorphism mapping $X=(H^1(\Omega))^2$ into its dual 
$X^*=(H^{1}(\Omega)')^2$. 
The canonical
choice of a preconditioner is therefore given as the blockdiagonal operator
\begin{align*}
  \mathcal{M}_{v}=(I+stJ^*J)^{-1}.
\end{align*}

The corresponding continuous linear operator from~\eqref{eq:KTK-uw} is
  \begin{align*}
    \MA_{u,w}= I+stK^*K=I+st\begin{pmatrix}H^*H & -H^*\symgrad\\-\symgrad^*H & \symgrad^*\symgrad\end{pmatrix},
  \end{align*}
  where $H=\symgrad(\nabla)$ is the Hessian matrix.
  In our experiments we chose the discrete version of the
  blockdiagonal operator
\begin{align*}
  \mathcal{M}_{u,w}=\begin{pmatrix} (I+stH^*H)^{-1} & 0\\ 0 &
    \left(I+ st J^*J\right)^{-1}\end{pmatrix}
\end{align*}
for preconditioning. It gives good numerical results, but is not as nicely accompanied by the theory
as the previous operator.
To obtain fast algorithms the inverses in the preconditioners can be well approximated by the incomplete
Cholesky decomposition.\footnote{We used the MATLAB function \texttt{ichol} with default parameters in the experiments. Varying the parameters did not lead to significantly different results.} 
For the denoising problems with the primal variables $u$ und $v$ we use the preconditioner
\begin{align}
  \label{preconditioner_uv}
  C_{u,v}=\begin{pmatrix} \operatorname{ichol}(I-st\Delta) & 0\\ 0 &
    \operatorname{ichol}\left(I+ st\left(I
    +J^TJ\right)\right)\end{pmatrix}
\end{align}
and for the denoising problems with the primal variables $u$ and $w$ we use the preconditioner
\begin{align}
  \label{preconditioner_uw}
  C_{u,w}=\begin{pmatrix} \operatorname{ichol}(I + H^TH) & 0\\ 0 &
    \operatorname{ichol}\left(I+ stJ^TJ\right)\end{pmatrix}.
\end{align}
For problem~\eqref{eq:DGTGV_1} we use the preconditioner
\begin{align}
  \label{preconditioner_v}
  C_{v}= \operatorname{ichol}(I + stJ^TJ).
\end{align}
Here the preconditioners are given in the form $M^{-1}=C^TC$.

\section{Experiments}
\label{sec:experiments}

In the previous sections we introduced several different models for variational image denoising and also described two algorithms, each applicable to each model.
This leaves us with a large number of parameters that have to be chosen.
In this section, we try to illustrate the effects of these parameters and to give some guidance on how these parameters should be chosen.
To that end, we divide the set of parameters into two groups:
\begin{description}
\item[Problem parameters:] These are the parameters of the model itself. For example in the case of TGV denoising~\eqref{eq:minTGV}, the two problem parameters are the two regularization parameters $\alpha_{0}$ and $\alpha_{1}$ while for MTGV~\eqref{eq:mtgv} we have the parameters $\alpha$ and $\delta_{1}$ and the DGTGV method consisting of~\eqref{eq:DGTGV_1} and~\eqref{eq:DGTV_2} also has the two parameters  $\alpha$ and $\delta_{1}$.
\item[Algorithmic parameters:] These are parameters, that influence only the algorithm, but not the theoretical minimizers. These can be for example: One or more step-sizes, the relaxation parameter, the stopping criterion (e.g. a tolerance for the duality gap), the parameters of the CG iteration, or the preconditioner.
\end{description}
The problem parameters influence the quality of the denoising, while the algoritmic parameters influence the performance (or speed) of the method. Moreover, there is a trade-off between speed and quality: if the algorithm is stopped too early, the minimizer may not be approximated well. Note, however, that sometimes early stopping may increase reconstruction quality (which often indicates that the model can be improved), but we shall not  deal with this question here but rather focus on the analysis of the problem and algorithmic parameters separately.

\subsection{Optimal constants in DGTV, DGTGV and MTGV}

The two-stage denoising method DGTV from section~\ref{sec:denoising-of-gradients} use the parameter $\delta_{2}$ which we motivated to be $\delta_{2} = c\norm[2]{\abs{\grad u^{0}}}$ in section~\ref{subsec:parameter-choice}. We calculated the optimal constants $c$ for various images with different noise-levels, cf. figure~\ref{fig:dgtv_c}.
As expected,  all $c\geq 1$ lead to the same result (in this case $v=0$ is a feasible solution and optimal and therefor, the two-stage method becomes pure TV-denoising).
If we choose $c<1$, we transfer a bit of structure of the input image into the gradient as an additional information for the image denoising step. Hence, $c\approx 0.99$ seems like a sensible choice.

\begin{figure}[H]
\begin{center}
	\includegraphics[scale=1]{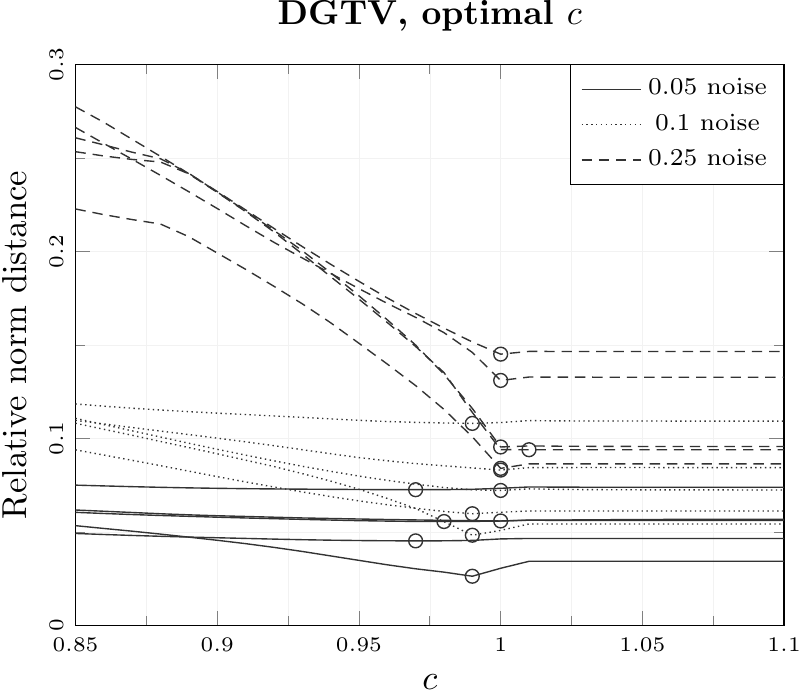}
	\caption{Optimal values for various images and different noise-level for the DGTV method.}\label{fig:dgtv_c}
\end{center}
\end{figure}

A similar experiment for the parameter $\alpha$ in the DGTGV method from section~\ref{sec:denoising-of-gradients} revealed  that all these optimal values are close to $1$, cf.~figure~\ref{fig:dgtgv_alpha_noise}. 
For $\alpha > 1$, the change in the norm distance is minimal, 
hence, the denoised image will be similar to the denoised image 
with $\alpha = 1$. Smaller values $\alpha < 1$ lead to worse 
reconstructions. Therefore, we use a default value of $\alpha = 1$.
For the MTGV method from section~\ref{sec:ctgv} we report the results of the optimization of the parameter $\alpha$ in figure~\ref{fig:mtgv_alpha_noise} and we see that values around $\alpha=2$ seem optimal (while the variance is larger than for DGTGV).

\begin{figure}[htb]
\begin{center}
\begin{subfigure}[htb]{\textwidth}
	\includegraphics[scale=1]{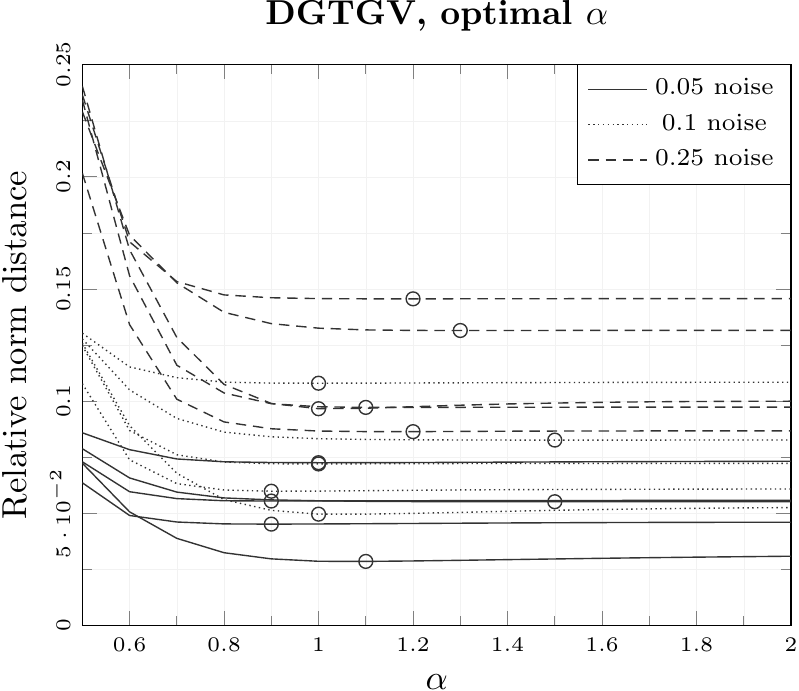}
\end{subfigure}
\subcaption{Optimal $\alpha$ for various images with different noise-levels in the DGTGV method.}\label{fig:dgtgv_alpha_noise}
\vspace{2ex}

\begin{subfigure}[htb]{\textwidth}
	\includegraphics[scale=1]{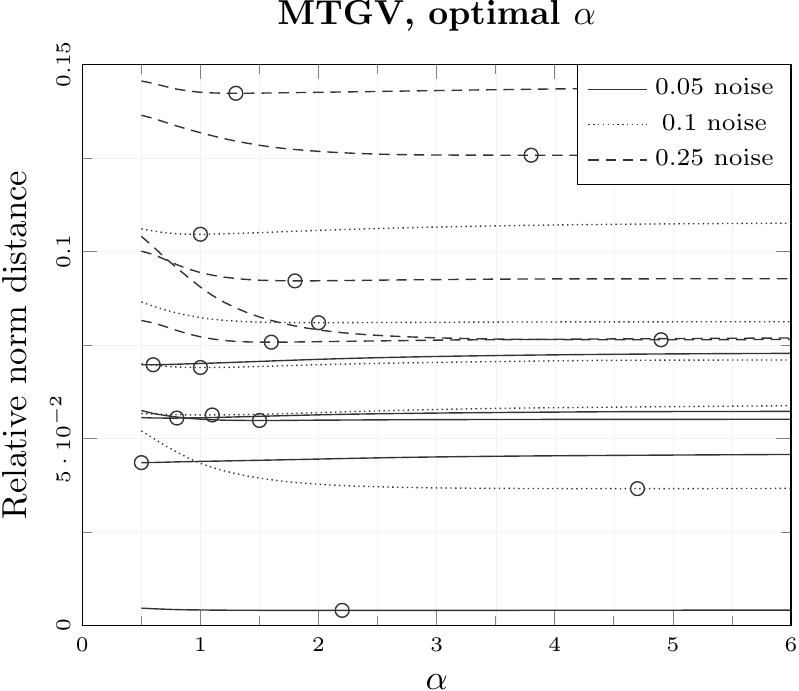}
\end{subfigure}
\subcaption{Optimal $\alpha$ for various images with different noise-levels in the MTGV method.}\label{fig:mtgv_alpha_noise}
\end{center}
\caption{Optimal $\alpha$ values according to relative norm distance between $u$ and the original image $\udag$ for various images and different noise-levels.}
\end{figure}

\subsection{Quality and runtime of MTGV and DGTGV}

Table~\ref{table:MTGVvsDGTGVdefault_sec_all} collects results on the reconstruction quality (measured in PSNR) and the runtime of the two-stage method DGTGV and the MTGV method.
In table~\ref{table:MTGVvsDGTGVbest} we compared the two-stage method DGTGV and the MTGV denoising method with the best possible $\alpha$ (i.e. we calculated an optimal value of this parameter for each image, noise-level and method) and the default $\delta_{1}$ as estimated from the images.
In table~\ref{table:MTGVvsDGTGVdefault} we compared both methods with the default $\alpha$ values, i.e. $\alpha = 1$ for DGTGV and $\alpha = 2$ for MTGV (again with $\delta_{1}$ estimated from the image). 

In both comparisons the PSNR values differ only slightly from each other showing that the default values always lead to results close to the optimal ones.
Moreover, the two-stage method DGTGV consistently gives slightly lower PSNR values than MTGV.
The difference in PSNR is so small, that the resulting images are very similar to each other (cf. figures~\ref{fig:DGTGV_MTGV_affine_default_01},~\ref{fig:DGTGV_MTGV_affine_best_025},~\ref{fig:DGTGV_MTGV_eye_best} for some examples). 

\begin{table*}[htb]                                                                   
\centering
\begin{subtable}[t]{0.2\linewidth}
\centering                                                                           
\begin{tabular}{l}                                                                 
\toprule                                                                             
Image(noise factor) \\                                                
\phantom{c}\\
\midrule                                                                             
affine(0.05) \\                                                  
affine(0.1) \\                                                   
affine(0.25) \\                                                  
\midrule
eye(0.05)\\                                                     
eye(0.1) \\                                                      
eye(0.25)\\                                                     
\midrule
cameraman(0.05) \\                                               
cameraman(0.1) \\                                                 
cameraman(0.25)  \\                                               
\midrule
moonsurface(0.05)\\                                             
moonsurface(0.1)  \\                                              
moonsurface(0.25) \\                                             
\midrule
barbara(0.05) \\                                                 
barbara(0.1) \\                                                   
barbara(0.25) \\                                                 
\bottomrule                                                                          
\end{tabular}                                                                        
\end{subtable}%
\hspace*{-0.35em}
\begin{subtable}[t]{0.15\linewidth}
\setlength\tabcolsep{2pt}
\centering                                                                           
\begin{tabular}{cc}                                                                 
\toprule                                                                             
DGTGV & MTGV \\   
\multicolumn{2}{c}{best $\alpha$}\\                  
\midrule                                                                             
37.26 & 38.30 \\                                                  
32.58 & 33.97 \\                                                   
27.74 & 28.41 \\                                                  
\midrule
31.34 & 31.64 \\                                                     
28.94 & 29.37 \\                                                      
26.26 & 26.98 \\                                                     
\midrule
30.57 & 30.79 \\                                               
27.14 & 27.32 \\                                                 
23.20 & 23.34 \\                                               
\midrule
30.69 & 30.87 \\                                             
28.46 & 28.73 \\                                              
26.08 & 26.38 \\                                             
\midrule
28.35 & 28.66 \\                                                 
24.91 & 25.14 \\                                                   
22.34 & 22.48 \\                                                 
\bottomrule                                                                          
\end{tabular}                                                                        
\caption{PSNR values for best possible $\alpha$ according to each method.}
\label{table:MTGVvsDGTGVbest}                                                        
\end{subtable}%
\hspace*{0.5em}
\begin{subtable}[t]{0.15\linewidth}
\centering                                                                           
\setlength\tabcolsep{2pt}
\begin{tabular}{cc}                                                                  
\toprule                                                                              
DGTGV & MTGV \\  
\multicolumn{2}{c}{default $\alpha$}\\                                                                 
\midrule                                                                              
37.07 & 38.25 \\                                                   
32.66 & 33.65 \\                                                    
26.80 & 27.87 \\                                                   
\midrule                                                                              
31.32 & 31.49 \\                                                       
28.95 & 29.28 \\                                                       
26.09 & 26.89 \\                                                      
\midrule                                                                              
30.53 & 30.78 \\                                                
27.09 & 27.32 \\                                                  
23.15 & 23.26 \\                                                
\midrule                                                                              
30.68 & 30.77 \\                                               
28.46 & 28.65 \\                                               
26.04 & 26.36 \\                                              
\midrule                                                                              
28.35 & 28.49 \\                                                  
24.91 & 25.06 \\                                                    
22.34 & 22.46 \\                                                  
\bottomrule                                                                           
\end{tabular}                                                                         
\caption{PSNR values for default $\alpha$ values according to each method.}
\label{table:MTGVvsDGTGVdefault}                                                      
\end{subtable}%
\hspace*{1em}
\begin{subtable}[t]{0.25\linewidth}
\centering
\setlength\tabcolsep{1.5pt}
\begin{tabular}{cccc}                                                                    
\toprule                                                                                  
{DGTGV} & {DGTGV} & 
{DGTGV} & {MTGV} \\                             
\multicolumn{1}{c}{(grad.)} & \multicolumn{1}{c}{(img.)} & 
\multicolumn{1}{c}{(all)} & \\
\midrule                                                                                  
0.12 & 0.06 & 0.17 & 0.20 \\                                                      
0.15 & 0.05 & 0.19 & 0.25 \\                                                     
0.27 & 0.04 & 0.31 & 0.39 \\                                                     
\midrule
0.14 & 0.16 & 0.30 & 0.91 \\                                                      
0.23 & 0.18 & 0.41 & 0.93 \\                                                      
0.67 & 0.16 & 0.83 & 1.53 \\                                                       
\midrule
0.30 & 0.15 & 0.46 & 1.28 \\                                                       
0.38 & 0.16 & 0.54 & 1.38 \\                                                        
0.73 & 0.16 & 0.89 & 1.81 \\                                                       
\midrule
0.14 & 0.16 & 0.30 & 0.84 \\                                                      
0.28 & 0.17 & 0.45 & 0.93 \\                                                      
0.65 & 0.15 & 0.80 & 1.52 \\                                                       
\midrule
1.19 & 0.58 & 1.77 & 10.98 \\                                                         
1.03 & 0.60 & 1.63 & 7.46 \\                                                         
1.94 & 0.75 & 2.69 & 7.33 \\     
\bottomrule                                                                               
\end{tabular}                                                                             
\caption{Time in seconds (Chambolle-Pock).}
\label{table:MTGVvsDGTGVdefault_time}                                                      
\end{subtable}%
\hspace*{1em}
\begin{subtable}[t]{0.15\linewidth}
\centering
\setlength\tabcolsep{1.5pt}
\begin{tabular}{cc}                                                                                       
\toprule                                                                                                     
DGTGV & MTGV \\
\multicolumn{2}{c}{ }\\
\midrule
0.18 & 0.28 \\                                                                                           
0.20 & 0.30 \\                                                                                           
0.44 & 0.34 \\                                                                                           
\midrule
0.46 & 1.21 \\                                                                                            
0.83 & 1.08 \\                                                                                              
1.66 & 1.26 \\                                                                                              
\midrule
0.80 & 1.79 \\                                                                                            
1.03 & 1.55 \\                                                                                             
1.51 & 1.43 \\                                                                                             
\midrule
0.52 & 1.17 \\                                                                                            
0.72 & 1.08 \\                                                                                             
1.46 & 1.26 \\                                                                                             
\midrule
3.40 & 9.67 \\                                                                                             
4.66 & 7.53 \\                                                                                             
8.03 & 6.04 \\
\bottomrule   
\end{tabular}                                                                                                
\caption{Time in seconds (Douglas-Rachford).}\label{table:DR_MTGVvsDGTGVdefault_sec}                                                                      
\end{subtable}
\caption{\ref{table:MTGVvsDGTGVbest},~\ref{table:MTGVvsDGTGVdefault}: PSNR values of DGTGV and MTGV methods with best possible $\alpha$ value and default values for each method;~\ref{table:MTGVvsDGTGVdefault_time}: Time in seconds for DGTGV and MTGV methods implemented with Chambolle-Pock, for the DGTGV both steps also separately;~\ref{table:DR_MTGVvsDGTGVdefault_sec}: Time in seconds for both methods implemented with Douglas Rachford.}\label{table:MTGVvsDGTGVdefault_sec_all}                                      
\end{table*} \hfill 

In tables~\ref{table:MTGVvsDGTGVdefault_time} and~\ref{table:DR_MTGVvsDGTGVdefault_sec}~we compare different run-times:
First, we used Chabolle-Pock's primal dual method for both steps of the DGTGV method separately, i.e. the gradient denoising as a pre-step, after that the actual image denoising and MTGV in table~\ref{table:MTGVvsDGTGVdefault_time}. It can be seen that the two-stage method DGTGV is faster (usually by a factor of 2 or 3).
In all cases we used a tolerance value of $10^{-3}$ for the relative primal-dual gap as a stopping criterion.
Table~\ref{table:DR_MTGVvsDGTGVdefault_sec} also shows runtimes for inexact preconditioned Douglas-Rachford method. Here we set the tolerance for the primal-dual gap to $10^{-2}$ since, this already gave better or comparable PSRN values (a phenomenon which we can not explain theoretically). However, even with this larger tolerance, the DR-method is only faster for large noise level and the MTGV method.

\begin{figure}[htb]
\begin{center}
\begin{subfigure}[htb]{0.45\linewidth}
	\includegraphics[width=1\linewidth]{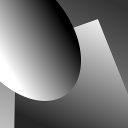}
	\caption{$\udag$}
\end{subfigure}\hspace{1em}%
\begin{subfigure}[htb]{0.45\linewidth}%
	\includegraphics[width=1\linewidth]{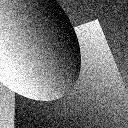}
	\caption{$u_0$}	
\end{subfigure}%

\begin{subfigure}[htb]{0.45\linewidth}
	\includegraphics[width=1\linewidth]{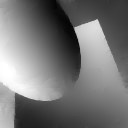}
	\caption{DGTGV}
\end{subfigure}\hspace{1em}%
\begin{subfigure}[htb]{0.45\linewidth}
	\includegraphics[width=1\linewidth]{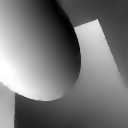}
	\caption{MTGV}
\end{subfigure}
\caption{Comparison of DGTGV and MTGV method of affine image with default $\alpha$ value; noise factor $0.1$.}
\label{fig:DGTGV_MTGV_affine_default_01}
\end{center}
\end{figure}

\begin{figure}[htb]
\begin{center}
\begin{subfigure}[htb]{0.45\linewidth}
	\includegraphics[width=1\linewidth]{affine.png}
	\subcaption{$\udag$}
\end{subfigure}\hspace{1em}%
\begin{subfigure}[htb]{0.45\linewidth}
	\includegraphics[width=1\linewidth]{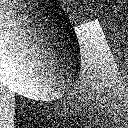}
	\subcaption{$u_0$}
\end{subfigure}%

\begin{subfigure}[htb]{0.45\linewidth}
	\includegraphics[width=1\linewidth]{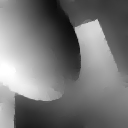}
	\subcaption{DGTGV}
\end{subfigure}\hspace{1em}%
\begin{subfigure}[htb]{0.45\linewidth}
	\includegraphics[width=1\linewidth]{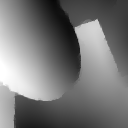}
	\subcaption{MTGV}
\end{subfigure}
\caption{Comparison of DGTGV and MTGV method of affine image with best possible $\alpha$ value; noise factor $0.25$.}
\label{fig:DGTGV_MTGV_affine_best_025}
\end{center}
\end{figure}

\begin{figure}[htb]
\begin{center}
\begin{subfigure}[htb]{0.45\linewidth}
	\includegraphics[width=1\linewidth]{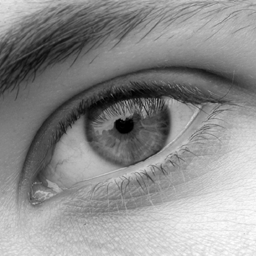}
	\subcaption{$\udag$}
\end{subfigure}\hspace{1em}%
\begin{subfigure}[htb]{0.45\linewidth}
	\includegraphics[width=1\linewidth]{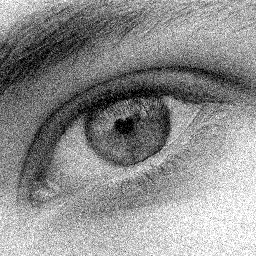}
	\subcaption{$u_0$}
\end{subfigure}\\%
\begin{subfigure}[htb]{0.45\linewidth}
	\includegraphics[width=1\linewidth]{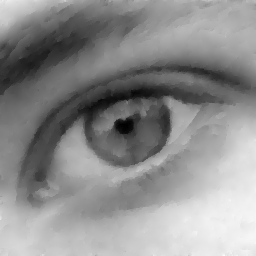}
	\subcaption{DGTGV}
\end{subfigure}\hspace{1em}%
\begin{subfigure}[htb]{0.45\linewidth}
	\includegraphics[width=1\linewidth]{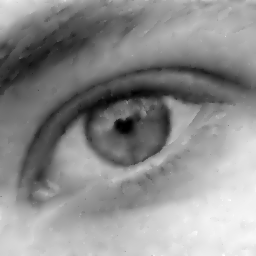}
	\subcaption{MTGV}
\end{subfigure}%
\caption{Comparison of DGTGV and MTGV method with best possible $\alpha$ value; noise factor $0.1$.}
\label{fig:DGTGV_MTGV_eye_best}
\end{center}
\end{figure}
 
\subsection{Comparison of different formulations}

The TGV minimization problem~\eqref{eq:minTGV}, the MTGV minimization problem~\eqref{eq:mtgv}
and the MTGV minimization problem with variable change $\eqref{eq:mtgvvarchange}$
 have the same solution parameter
$u$ if we choose $\delta_1 = \norm[2]{u_0 - u^{\TGV}}$ and $\alpha=2$, where $u^{\TGV}$ is the solution of~\eqref{eq:minTGV}. For the first two problems, we used the algorithms of Douglas-Rachford and Chambolle-Pock. 
The latter one is 
an $\mathcal{O}(1/k)$ primal-dual algorithm with constant step sizes (see
\cite{chambolle2011first}), namely
  \begin{align*}
    %\label{CP}
    %\tag{CP}
    \begin{split}
      y^{n+1}&=(I+\sigma\partial G^*)^{-1}(y^n+\sigma K(2x^{n}-x^{n-1}))\\
        x^{n+1}&=(I+\tau\partial F)^{-1}(x^n-\tau K^*y^{n+1})
\end{split}
\end{align*}
For MTGV with variable change we only show the results for the Douglas-Rachford algorithm, since
the Chambolle-Pock algorithm was much slower for this problem.
In figure~\ref{fig:tgv_mtgv} we compare iteration number and time needed to obtain the
desired accuracy of the image. Since the duality gap is not suitable to compare different
minimization problems, the tolerance is given by
$\frac{\norm[2]{u-u^{\TGV}}}{\norm[2]{u^{\TGV}}}$. The reference value
$u^{\TGV}$ is obtained by solving~\eqref{eq:minTGV} with 1000000 iterations of
Chambolle-Pock ($\tau=0.004$,  $\sigma=\frac{1}{\tau\|K\|^2}$, $\|K\|^2=12$). 
The algorithms are tested with the image eye (256x256) corrupted by Gaussian white
noise of mean $0$ and variance $0.1$. The stepsizes of the algorithms are chosen by trial and error:
\begin{itemize}
  \item CP TGV: $\tau=0.008$, $\sigma=\frac{1}{\tau\|K\|^2}$, $\|K\|^2=12$,
  \item CP MTGV: $\tau=0.004$, $\sigma=\frac{1}{\tau\|K\|^2}$, $\|K\|^2=12$,
  \item DR TGV: $s=60$, $t=0.28$,
  \item DR MTGV: $s=60$, $t=0.1$,
  \item DR MTGV (var. change): $s=60$, $t=0.04$.
\end{itemize}
The optimal stepsize depends on the accuracy needed. The Douglas-Rachford algorithm is
used in an inexact manner. The linear operator is approximated by two iterations of the
preconditioned conjugate gradients method, where the preconditioners are given by
$\eqref{preconditioner_uv}$ and $\eqref{preconditioner_uw}$. In figure~\ref{fig:tgv_mtgv} we can see that
the algorithms for TGV and MTGV are competitive. The variable change leads to much slower algorithms.

\begin{figure}[htb]
\hspace{-4em}
\begin{subfigure}{\linewidth}
	\includegraphics[scale = 1]{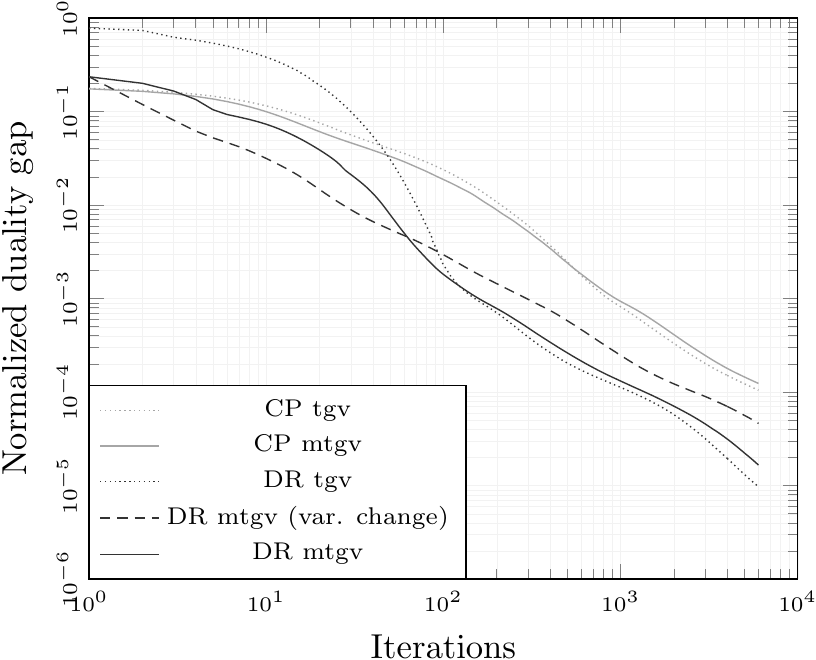}
\end{subfigure}
\begin{subfigure}{\linewidth}
	\includegraphics[scale = 1]{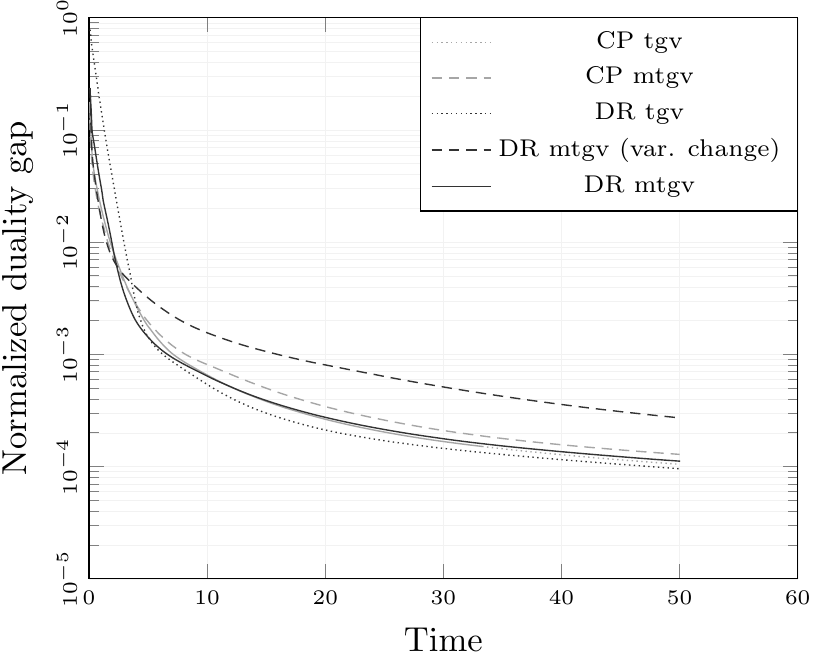}
\end{subfigure}
\caption{Runtime TGV vs. MTGV.}
  \label{fig:tgv_mtgv}
\end{figure}

\subsection{Inexactness for the Douglas-Rachford method}
The Douglas-Rachford iteration in general allows inexact evaluation of the operators as long as the
error stays sum\-mable, see e.g.~\cite{combettes2004averaged}. We made experiments with MTGV, using a few iterates of the conjugate gradient 
method with and without preconditioning (see figure \ref{fig:cg_pcg_iter_time}).The algorithms are tested 
with the image eye (256x256) corrupted by Gaussian white
noise of mean $0$ and variance $0.1$. The stepsizes of
the Douglas-Rachford iteration are chosen by trial and error: $s=120$, $t=0.1$. The Douglas-Rachford iteration is very slow for one or two
iterations of CG without preconditioning. It does not even converge for three iterations of CG
without
preconditioning. Using the preconditioners as proposed in section~\ref{sec:idr} we obtain very good
convergence of Douglas-Rachford for any number of iterations of PCG. In figure~\ref{fig:cg_pcg_iter_time} we can see that
preconditioning is crucial for the method to converge. After 300 iterations all inexact algorithms
have caught up with the exact one regarding the iteration count, while only one or two iterations perform best regarding the computational time.

\begin{figure*}[htb]
\begin{center}
\begin{subfigure}[htb]{\columnwidth}
	\includegraphics[scale=1]{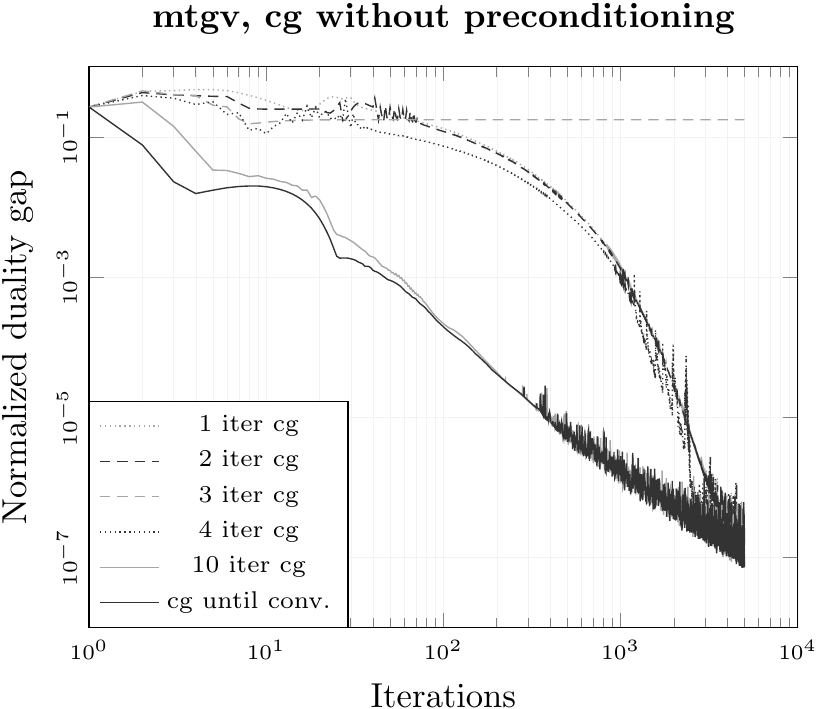}
	\subcaption{Convergence of MTGV without preconditioning according to number of iterations.}\label{fig:cg_cg_iterations}%
\end{subfigure}%
\hspace*{2em}%
\begin{subfigure}[htb]{\columnwidth}
	\includegraphics[scale=1]{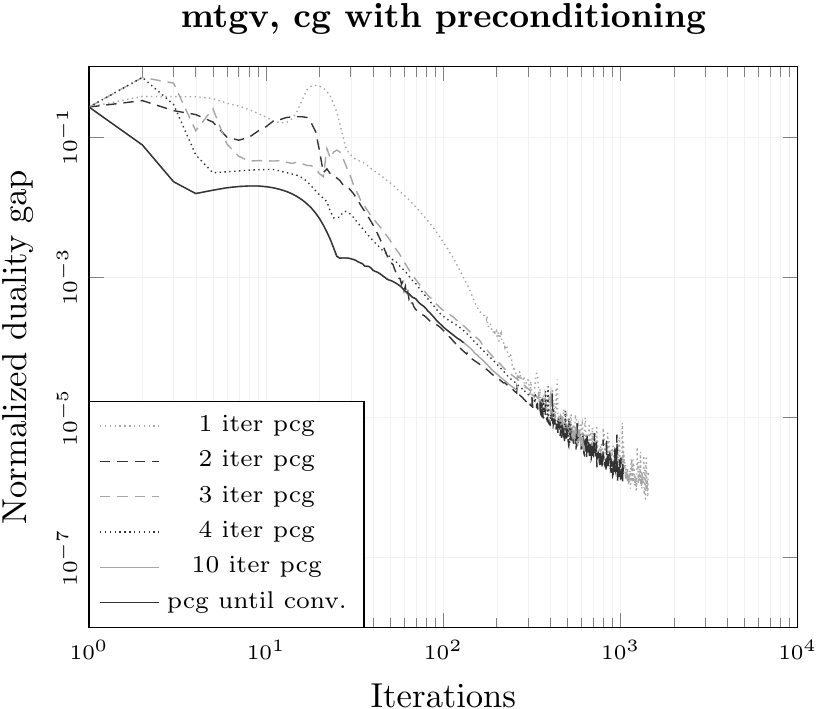}
	\subcaption{Convergence of MTGV with preconditioning according to number of iterations.}\label{fig:cg_pcg_iterations}%
\end{subfigure}
\vspace{2ex}

\begin{subfigure}{\columnwidth}
	\includegraphics[scale = 1]{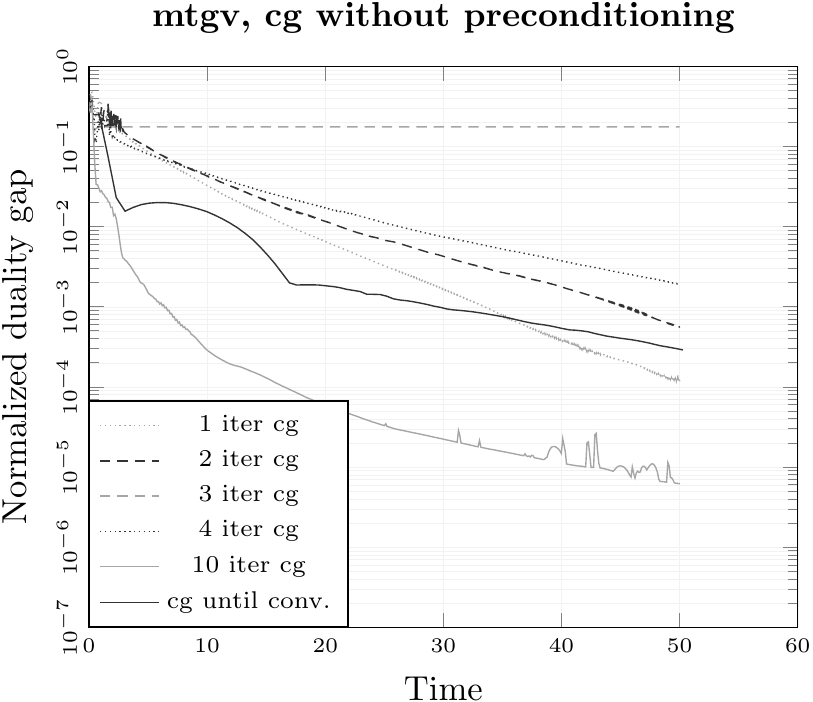}
	\subcaption{Convergence of MTGV without preconditioning according to time in seconds.}\label{fig:cg_cg_time}%
\end{subfigure}%
\hspace*{2em}%
\begin{subfigure}{\columnwidth}
	\includegraphics[scale = 1]{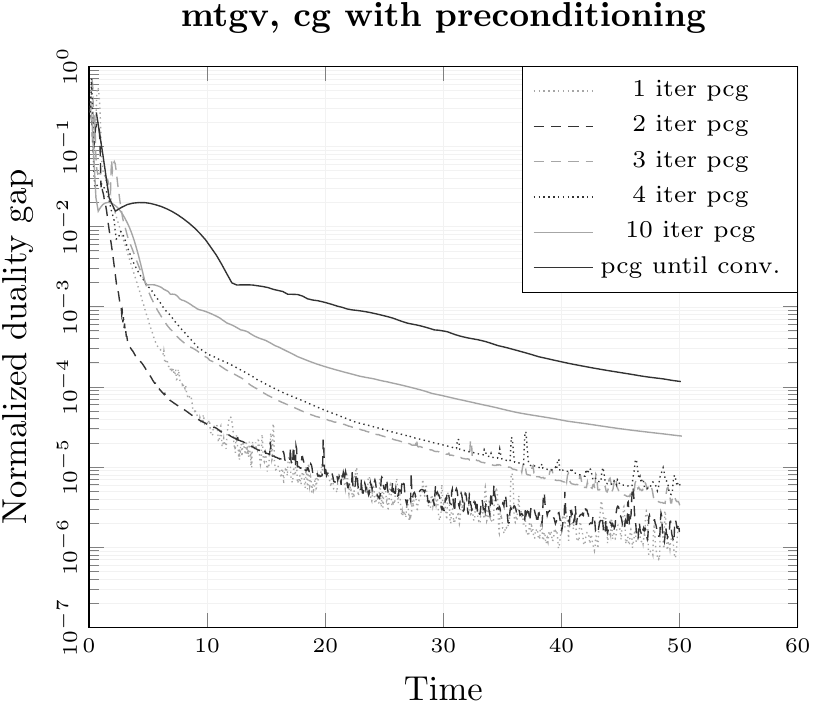}
	\subcaption{Convergence of MTGV with preconditioning according to time in seconds.}\label{fig:cg_pcg_time}%
\end{subfigure}
\end{center}
\caption{A comparisson of the MTGV method with and without preconditioning according to iterations and time.}\label{fig:cg_pcg_iter_time}
\end{figure*}

\section{Conclusion}
\label{sec:conclusion}
 We investigated variants of variational denoising methods using total-variation penalties and an estimate of the image gradient. First, this provides a natural and, at least to us, new interpretation of the successful TGV method.
The reformulation with a constraint for the discrepancy term $\norm[2]{u-u^{0}}$ (which also works for all other norms) together with our empirical observations allows for variational denoising methods that are basically parameter free and we mainly investigated the methods DGTGV and MTGV.
Since the gradient of $u^{0}$ is even more noisy that $u^{0}$ itself, our experiments in Section~\ref{sec:experiments} show, that it is still useful to use a denoised version of $\nabla u^{0}$ as estimate for $\nabla u^{\dag}$. Indeed,  we obtained that the two-stage method DGTGV sacrifices only little denoising performance for a substantial gain in speed.
Put differently, solving the two denoising problems for the gradient and the image is significantly easier, than solving the combined MTGV problem and the denoising result is still good.
This qualifies the two-stage DGTGV as an alternative to MTGV (and hence, TGV) as a denoising method.

Another part of the investigation involved the Douglas-Rachford method for these problems. Here we could derive natural preconditioners for the linear sub-problems and show that they greatly improve the overall speed of the method, especially in the inexact case. However, the overall runtime was only better than the simpler primal-dual method by Chambolle and Pock in case of large noise.

\appendix
\section{Appendix}

\subsection{Discretization}
\label{sec:discretization}
% differential operators, norms, inner product, projections

We equip the space of  $M\times N$ images with the inner product and induced norm
\begin{align*}
  &\langle u,v\rangle=\sum_{i=1}^M\sum_{j=1}^N u_{i,j}v_{i,j},\quad \|u\|_2
  =\left(\sum_{i=1}^M\sum_{j=1}^N u_{i,j}^2\right)^{\frac{1}{2}}.
\end{align*}
 For $u\in\RR^{M\times N}$ we define the discrete partial forward derivative (with constant boundary
 extension $u_{M+1,j}=u_{M,j}$ and $u_{i,N+1}=u_{i,N}$ as
  \begin{align*}
    (\partial_1u)_{i,j}&=u_{i+1,j}-u_{i,j}, &
    (\partial_2u)_{i,j}&=u_{i,j+1}-u_{i,j}.
  \end{align*}
  The discrete gradient $\nabla:\RR^{M\times N}\to\RR^{M\times N\times 2}$ is defined by
  \begin{align*}
    (\nabla u)_{i,j,k}=(\partial_ku)_{i,j}.
  \end{align*}
 The symmetrized gradient $\mathcal{E}$ maps from $\RR^{M\times N \times 2}$ to $\RR^{M\times N\times
4}$. For simplification of notation, the 4 blocks are written in one plane:
\begin{align*}
  \mathcal{E}v&=\frac{1}{2}\left(\nabla v+(\nabla v)^T\right)\\
  &=
  \begin{pmatrix}
    \partial_1v_1 & \frac{1}{2}(\partial_1v_2+\partial_2v_1)\\
	\frac{1}{2}(\partial_1v_2+\partial_2v_1) & \partial_2v_2
      \end{pmatrix}.
\end{align*}
 
  The norm $\| |\cdot |\|_1$ in the space $\RR^{M\times N\times K}$ 
  reflects that for
  $v\in\RR^{M\times N \times K}$ we consider $v_{i,j}$ as a vector in $\RR^K$ on which we use
  the Euclidean norm:
  \begin{align*}
    &\||v|\|_1\coloneqq\sum_{i=1}^M\sum_{j=1}^N|v_{i,j}|\quad\text{with
    }|v_{i,j}|\coloneqq\left(\sum_{k=1}^Kv_{i,j,k}^2\right)^{\frac{1}{2}}.
  \end{align*}

  The discrete divergence is the negative adjoint of $\nabla$, i.e. the unique linear mapping
  $ \div: \RR^{M\times N \times 2}
  \to\RR^{M\times N}$, which satisfies
  \begin{align*}
    \langle\nabla u,v\rangle=-\langle u,\div v\rangle,\quad\text{for all }u,v.
  \end{align*}
  The adjoint of the symmetriced gradient is the unique linear mapping
  $ \symgrad^*: \RR^{M\times N \times 4}
  \to\RR^{M\times N \times 2}$, which satisfies
  \begin{align*}
    \langle\symgrad v,p\rangle=\langle v,\symgrad^* p\rangle,\quad\text{for all }v,p.
  \end{align*}

\subsection{Prox operators and duality gaps for considered problems}
\label{sec:proxdual}
In order to calculate experiments with the methods proposed in the previous 
sections, in this section we will state all 
primal and dual functionals according to a general optimization 
problem $\min_x F(x) + G(Kx)$ along with a study of the 
corresponding primal-dual gaps and possibilities to ensure 
feasibility of the iterates throughout the program by 
reformulation of the problems by introducing a substitution 
variable. We also give the proximal operators needed for the Chambolle-Pock and Douglas-Rachford
algorithm.
\subsubsection{DGTV}
\label{subsec:DGTV}
In section~\ref{sec:two-stage-denoising} we formulated a two-
staged denoising method in two ways. 
First, as a constrained version~\eqref{eq:DGTV_1} and~
\eqref{eq:DGTV_2}. 
In this formulation we get the primal functionals for the first problem~\eqref{eq:DGTV_1}
\begin{align}\label{eq:primalDGTVconstraint}
\begin{split}
	F(v) &= \ind_{\norm[1]{\abs{\grad u_0 -\dummy}}
		\leq \delta_2}(v),\\
	G(\psi) &= \norm[1]{\abs{\psi}}
\end{split}
\end{align}
with operator $K = \symgrad$.
The dual problems, in general written as 
\begin{align*}\max_y -F^*(-K^*y)-G^*(y)\end{align*} are
\begin{align}\label{eq:dualDGTVconstraint}
\begin{split}
	F^*(t) &= \delta_2\norm[\infty]{\abs{t}} 
			+ \scp{t}{\grad u_0},\\
	G^*(q) &= \ind_{\norm[\infty]{\abs{\dummy}}\leq 1}(q)
\end{split}
\end{align}
with operator $K^* = \symgrad^*$.
The primal-dual gap writes as
\begin{align}\label{eq:gapDGTVconstraint}
\begin{split}
	\gap&_{\DGTV}^{(1)}(v,q) 
	= \ind_{\norm[1]{\abs{\grad u_0 - \dummy}}\leq 1}(v) 
		+ \norm[1]{\abs{\symgrad v}}\\
		&+ \delta_2\norm[\infty]{\abs{\symgrad^* q}}
		- \scp{\symgrad^* q}{\grad u_0}
		+ \ind_{\norm[\infty]{\abs{\dummy}}\leq 1}(q).
\end{split}
\end{align}
The proximal operators are
\begin{align}\label{eq:proxDGTVconstraint}
\begin{split}
	\prox_{\tau  F}(v) 
	&= \proj_{\norm[1]{\abs{\grad u_0 - \dummy}}
		\leq \delta_2}(v),\\
	\prox_{\sigma G^*}(q)
	&= \proj_{\norm[\infty]{\abs{\dummy}}\leq 1}(q).
\end{split}
\end{align}

For the second problem within DGTV, we have the denoising problem 
of the image with respect to the denoised gradient $\widehat{v}$ 
as output of 
the previous problem, cf. problem~\eqref{eq:DGTV_2}.
There, the primal functionals with $K = \grad$ are
\begin{align}\label{eq:primalDGTVu}
\begin{split}
	F(u)   &= \ind_{\norm[2]{u_0 - \dummy}\leq\delta_1}(u)\\
	G(\phi)&= \norm[1]{\abs{\phi - \widehat{v}}}
\end{split}
\end{align}
and the corresponding dual functionals write as 
\begin{align}\label{eq:dualDGTVu}
\begin{split}
	F^*(s) &= 	\delta_1\norm[2]{s} + \scp{u_0}{s}\\
	G^*(p) &= 	\ind_{\norm[\infty]{\abs{\dummy}}\leq 1}(p)
				+ \scp{\widehat{v}}{p}.
\end{split}
\end{align}
Hence, the primal-dual gap for this problem is
\begin{align}\label{eq:gapDGTVu}
\begin{split}
	\gap_{\DGTGV}^{(2)}&(u,p)
	= 	\ind_{\norm[2]{u_0-\dummy}\leq \delta_1}(u)
		+ \norm[1]{\abs{\grad u - \widehat{v}}}\\
		&- \scp{u_0}{\grad^* p}
		+ \ind_{\norm[\infty]{\abs{\dummy}}\leq 1}(p)
		+ \scp{\widehat{v}}{p}.
\end{split}	
\end{align}
The proximal operators are given by
\begin{align}\label{eq:proxDGTVu}
\begin{split}
	\prox_{\tau F}(u) 
	&= \proj_{\norm[2]{u_0 - \dummy}\leq \delta_1}(u)\\
	\prox_{\sigma G^*}(p) 
	&= \proj_{\norm[\infty]{\abs{\dummy}}\leq 1}(
		p - \sigma \widehat{v}).
\end{split}	
\end{align}

\subsubsection{DGTGV}
\label{subsec:DGTGV}

We reformulate problem~\eqref{eq:DGTGV_1} by using a substitution 
$w = \grad u_0 - v$, also considered in section~\ref{sec:gaps}, 
where we calculated another duality gap.
Hence, we get the primal functionals for the gradient denoising 
problem with operator $K = \symgrad$ as
\begin{align}\label{eq:primalDGTGVw}
\begin{split}
	F(w) 	&= \norm[1]{\abs{w}}\\
	G(\psi) &= \alpha\norm[1]{\abs{\symgrad \grad u_0 - w}}.
\end{split}
\end{align}
Therefore, the dual functionals write as
\begin{align}\label{eq:dualDGTGVw}
\begin{split}
	F^*(t) 	&= \ind_{\norm[\infty]{\abs{\dummy}}\leq 1}(t)\\
	G^*(q) 	&= \ind_{\norm[\infty]{\abs{\dummy}}\leq \alpha}(q)
				+ \scp{\symgrad\grad u_0}{q}.
\end{split}
\end{align}
With that the primal-dual gap is
\begin{align}\label{eq:gapDGTGVw}
\begin{split}
	\gap^{(1)}&_{\DGTGV}(w,q) = 
	 \norm[1]{\abs{w}} 
		+ \alpha \norm[1]{\abs{\symgrad \grad u_0 - w}} \\
		&+ \ind_{\norm[\infty]{\abs{\dummy}}\leq 1}(\symgrad^*q)
		+ \ind_{\norm[\infty]{\abs{\dummy}}\leq \alpha}(q)
		+ \scp{\symgrad\grad u_0}{q}.
\end{split}
\end{align}
The proximal operators are, with Moreau's identity for the first 
one, 
\begin{align}
\begin{split}
	\prox_{\tau  F}(w)
	&= w - \tau\prox_{\tau^{-1} F^{*}}(\tau^{-1}w)\\
	&= w - \tau\proj_{\norm[\infty]{\abs{\dummy}}\leq 1}
		(\tau^{-1}w)\\
	\prox_{\sigma G^*}(q)
	&= \proj_{\norm[\infty]{\abs{\dummy}}\leq \alpha}
		(q - \tau\symgrad\grad u_0).
\end{split}
\end{align}
For the second problem, we already derived all functionals, gaps and proximal operators in subsection~\ref{subsec:DGTV}, 
equations~\eqref{eq:primalDGTVu}--\eqref{eq:proxDGTVu}.

\subsubsection{CTGV}
\label{sec:dualctgv}
The Morozov type constrained total generalized variation denoising 
problem was formulated in section~\ref{sec:ctgv} 
(cf.~\eqref{eq:minCTGV}). For this formulation the primal functionals
are
\begin{align}\label{eq:primalmctgv}
%\begin{split}
%	F(u,v) &= \ind_{\norm[2]{u - u_0}\leq \delta_1}(u),\\
%	G(\phi, \psi) &= \ind_{\norm[1]{\abs{\dummy}}\leq \delta_1}
%					(\phi) + \norm[1]{\abs{\psi}}
%\end{split}
\begin{split}
	F(u,w) &= \ind_{\norm[2]{u - u_0}\leq \delta_1}(u) + 
		\ind_{\norm[1]{\abs{\dummy}}\leq \delta_2}(w),\\
	G(\phi) &= \norm[1]{\abs{\phi}}
\end{split}
\end{align}
with block operator 
\begin{align}\label{eq:operatormctgv}
	K = \symgrad\pmat{\grad & -\id}.
\end{align}
The corresponding dual functionals are
\begin{align}\label{eq:dualctgv}
\begin{split}
	F^*(s,t) &= \delta_1\norm[2]{s} + \scp{s}{u_0}
			+ \delta_2\norm[\infty]{\abs{t}},\\
	G^*(q) &= \ind_{\norm[\infty]{\abs{\dummy}}\leq 1}(q)
\end{split}
\end{align}
with dual block operator
\begin{align}\label{eq:dualoperatormctgv}
	K^* = \pmat{\grad^* \\ -\id}\symgrad^*.
\end{align}
The proximal operators are accordingly given by
\begin{align}\label{eq:proxmctgv}
\begin{split}
	&\prox_{\tau F}(u,w) 
	= \pmat{
		\proj_{\norm[2]{\dummy - u_0}\leq \delta_1}(u)\\
		\proj_{\norm[1]{\abs{\dummy}}\leq \delta_2}(w)
		},\\
	&\prox_{\sigma G^*}(q) 
	= \proj_{\norm[\infty]{\abs{\dummy}}\leq 1}(q)
\end{split}
\end{align}
and the primal-dual gap writes as
\begin{align}\label{eq:pdgapmctgv}
\begin{split}
	\gap&_{\operatorname{CTGV}}(u,w,q)=\\
   	&\norm[1]{\abs{\symgrad(\grad u - w)}} 
			+\ind_{\norm[2]{\dummy -u_0}\leq \delta_1}(u)
			+\ind_{\norm[1]{\abs{\dummy}}\leq \delta_2}(w)\\
			&+\delta_1\norm[2]{\grad^*\symgrad^*q} 
			- \scp{\grad^*\symgrad^*q}{u_0}\\
			&+\delta_2 \norm[\infty]{\abs{\symgrad^*q}} + 
			\ind_{\norm[\infty]{\abs{\dummy}}\leq 1}(q).
\end{split}
\end{align}

\subsubsection{MTGV}
\label{sec:dualmtgv}
In section~\ref{sec:ctgv} we defined the MTGV optimization 
problem~\eqref{eq:mtgv} as a sort of mixed version between the 
TGV~\eqref{eq:minTGV} and the CTGV~\eqref{eq:minCTGV} problems. Thus, 
the primal functionals are
\begin{align}\label{eq:primalmtgv}
\begin{split}
	F(u,v) &= \ind_{\norm[2]{\dummy - u_0}\leq \delta_1}(u),\\
	G(\phi, \psi) &= \norm[1]{\abs{\phi}} + \alpha
				\norm[1]{\abs{\psi}}
\end{split}
\end{align}
with the block operator
\begin{align}\label{eq:operatormtgv}
  K = \pmat{\grad & -\id\\0 &\symgrad}.
\end{align}
Accordingly, the dual functionals are
\begin{align}\label{eq:dualmtgv}
\begin{split}
	F^*(s,t) &= \delta_1\norm[2]{s} + \scp{s}{u_0}
			+ \ind_{\{0\}}(t),\\
	G^*(p,q) &= \delta_2\norm[\infty]{\abs{p}}
			+ \ind_{\norm[\infty]{\abs{\dummy}}\leq 1}(q)
\end{split}
\end{align}
and the dual operator
\begin{align}\label{eq:dualoperatormtgv}
  K = \pmat{\grad^* & 0\\-\id &\symgrad^*}.
\end{align}
The proximal operators are given by 
\begin{align*}
	&\prox_{s F}(u,v)=(\proj_{\norm[2]{\dummy-u_0}\leq \delta_1}(u),v)\\
	&\prox_{t G^*}(p,q)=(\proj_{\norm[\infty]{\abs{\dummy}}\leq
	1}(p),\proj_{\norm[\infty]{\abs{\dummy}}\leq \alpha}(q))
\end{align*}
and the gap function is given by
\begin{align*}
     \gap&_{\operatorname{MTGV}}(u,v,p,q)=\\
   &\norm[1]{\abs{\grad u - v}} + 
		\alpha\norm[1]{\abs{\symgrad(v)}}
		+\ind_{\{\norm[2]{\dummy -u_0}\leq \delta_1\}}(u)\\
		&+\delta_1\norm[2]{\grad^* p} - \langle p, \grad u_0\rangle 
		+\ind_{\{0\}}(p - \symgrad^*q) \\
		&+\ind_{\{\norm[\infty]{\abs{\dummy}}\leq\alpha\}}(p) +
		\ind_{\{\norm[\infty]{\abs{\dummy}}\leq 1\}}(q).
     \end{align*}
     To circumvent the feasibility problem, as introduced in \ref{sec:gaps} one can use the modified
     gap function
\begin{align*}
  \gap&_{\operatorname{MTGV}}(u,v,\tilde q)=\\
 &\leq \norm[1]{\abs{\grad u - v}} + 
		\alpha\norm[1]{\abs{\symgrad(v)}}
		+\ind_{\{\norm[2]{\dummy-u_0}\leq \delta_1\}}(u)\\
		&+\delta_1\norm[2]{\grad^* (\symgrad^*\tilde q)} - \langle \symgrad^*\tilde q, \grad
		u_0\rangle + \ind_{\{\norm[\infty]{\abs{\dummy}}\leq 1\}}(\tilde q),
		\end{align*}
		where $\tilde
		q\coloneqq\frac{q}{\max\left(1,\frac{\norm[\infty]{\abs{\symgrad^*q}}}{\alpha}\right)}$.

\subsubsection{TGV:}
For TGV~\eqref{eq:minTGV} we have the primal functionals
\begin{align*}
	&F(u,v)= \frac{1}{2}\norm[2]{u-u_0}^2\\
	&G(\phi,\psi)=\alpha_1\norm[1]{\abs{\phi}} + 
		\alpha_0\norm[1]{\abs{\psi}},
\end{align*}
with the same block operator~\eqref{eq:operatormtgv}. The corresponding dual functionals are
\begin{align*}
	&F^*(s,t) = \frac{1}{2}\norm[2]{s}^2 + \scp{s}{u_0}
			+ \ind_{\{0\}}(t),\\
	&G^*(p,q)=\ind_{\{\norm[\infty]{\cdot}\leq \alpha_1\}}(p)+ \ind_{\{\norm[\infty]{\abs{\dummy}}\leq
	\alpha_0\}}(q)
\end{align*}
and the dual operator is~\eqref{eq:dualoperatormtgv}. The proximal operators are given by
\begin{align*}
	&\prox_{t F}(u,v)=\left(\frac{u+tu_0}{1+t} ,v\right)\\
	&\prox_{t G^*}(p,q)=(\proj_{\norm[\infty]{\abs{\dummy}}\leq
	\alpha_1}(p),\proj_{\norm[\infty]{\abs{\dummy}}\leq \alpha_0}(q))
      \end{align*}
      and the gap function is given by
      \begin{align*}
     \gap&_{\TGV}(u,v,p,q)=
   \alpha_1\norm[1]{\abs{\grad u - v}} + 
		\alpha_0\norm[1]{\abs{\symgrad(v)}}\\
		&+\frac{1}{2}\norm[2]{u-u_0}^2
		+\frac{1}{2}\norm[2]{\grad^* p}^2 - \langle p, \grad u_0\rangle 
		+\ind_{\{0\}}(p - \symgrad^*q) \\
		&+\ind_{\{\norm[\infty]{\abs{\cdot}}\leq\alpha_0\}}(p) +
		\ind_{\{\norm[\infty]{\abs{\cdot}}\leq\alpha_1\}}(q).
	      \end{align*}
	      To circumvent the feasability problem, as introduced in  \ref{sec:gaps} one can use the modified
     gap function
	      \begin{align*}
     \gap&_{\TGV}(u,v,\tilde q)=\alpha_1\norm[1]{\abs{\grad u - v}} + 
		\alpha_0\norm[1]{\abs{\symgrad(v)}}\\
		&+\frac{1}{2}\norm[2]{u-u_0}^2
		+\frac{1}{2}\norm[2]{\grad^* (\symgrad^*\tilde q)}^2 - \langle \symgrad^*\tilde q, \grad
		u_0\rangle\\ 
		&+ \ind_{\{\norm[\infty]{\abs{\cdot}}\leq\alpha_1\}}(\tilde q).
	      \end{align*}
		where $\tilde
		q\coloneqq\frac{q}{\max\left(1,\frac{\norm[\infty]{\abs{\symgrad^*q}}}{\alpha}\right)}$.

\subsection{Projections}
The projections used for the algorithms are
   \begin{align*}
     \proj_{\norm[2]{\cdot-u_0}\leq \delta}(u)&=u_0 +
     \max\left(\frac{\delta}{\norm[2]{u-u_0}},1\right)\cdot(u-u_0),\\
%    \proj_{\norm[1]{\abs{\cdot}}\leq \delta}(v)&= \text{ Birgit? }\\
     \proj_{\norm[\infty]{\abs{\cdot}}\leq\delta}(v)&=\frac{v}{\max(1,\frac{\abs{v}}{\delta})},
   \end{align*}
   with $\abs{v}=\left(\sum_{k=1}^dv_k^2\right)^{\frac{1}{2}}$ in the pointwise sense.
   
The idea how to project onto an mixed norm ball, i.e. $\norm[1]{\abs{\dummy}}$, can be found with in~\cite{Songsiri2011}. There the author developed an algorithm to project onto an $l_1$-norm ball and after that onto a sum of $l_2$-norm balls. The projection itself cannot be stated in a simple closed form.

\end{document}